\documentclass[a4paper,11pt]{article}
\usepackage{lmodern} 
\usepackage[T1]{fontenc} 
\usepackage[english]{babel} 
\usepackage[justification=centering]{caption, subcaption} 
\usepackage{graphicx}
\usepackage{bbm}
\usepackage{array}
\usepackage{hyperref}
\usepackage{xcolor}
\usepackage{amsthm}
\usepackage{stmaryrd}
\usepackage{enumerate}
\usepackage[a4paper,top=2.54cm,bottom=2.54cm,left=2.54cm,right=2.54cm,marginparwidth=1.75cm]{geometry}

\newtheorem{theorem}{Theorem}[section]
\newtheorem{corollary}[theorem]{Corollary}
\newtheorem{lemma}[theorem]{Lemma}
\newtheorem{definition}[theorem]{Definition}
\newtheorem{proposition}[theorem]{Proposition}

\theoremstyle{definition}
\newtheorem{remark}[theorem]{Remark}

\usepackage{amssymb}
\usepackage{latexsym,amsmath,amssymb}

\newcommand{\eps}{\varepsilon}
\newcommand{\ind}{\mathbbm{1}}
\renewcommand{\P}{\mathbb{P}}
\renewcommand{\ge}{\geqslant}
\newcommand{\E}{\mathbb{E}}
\newcommand{\R}{\mathbb{R}}
\newcommand{\Z}{\mathbb{Z}}

\newcommand{\Hor}{\mathrm{Hor}_{x,y}^{(n)}}
\newcommand{\LVer}{\mathrm{LVer}_{x,y}^{(n)}}
\newcommand{\RVer}{\mathrm{RVer}_{x,y}^{(n)}}

\DeclareMathOperator{\br}{br}

\author{Diana De Armas Bellon\thanks{Department of Mathematical Sciences, University of Bath, Bath BA2 7AY, UK. Email: \texttt{ddabda20@bath.ac.uk}}, Matthew I.~Roberts\thanks{Department of Mathematical Sciences, University of Bath, Bath BA2 7AY, UK. Email: \texttt{mattiroberts@gmail.com}}}

\title{Accessibility Percolation with Rough Mount Fuji labels}

\begin{document}

\flushbottom
\maketitle

\begin{abstract}
Consider an infinite, rooted, connected graph where each vertex is labelled with an independent and identically distributed Uniform(0,1) random variable, plus a parameter $\theta$ times its distance from the root $\rho$. That is, we label vertex $v$ with $X_v = U_v + \theta d(\rho,v)$. We say that \emph{accessibility percolation} occurs if there is an infinite path started from $\rho$ along which the vertex labels are increasing.

When the graph is a Bienaym\'e-Galton-Watson tree, we give an exact characterisation of the critical value $\theta_c$ such that there is accessibility percolation with positive probability if and only if $\theta>\theta_c$. We also give more explicit bounds on the value of $\theta_c$. The lower bound holds for a much more general class of trees.

When the graph is the lattice $\mathbb{Z}^n$ for $n\ge 2$, we show that there is a non-trivial phase transition and give some first bounds on $\theta_c$. To do this we introduce a novel coupling with oriented percolation.
\end{abstract}

\section{Introduction}
Accessibility percolation is a probabilistic model motivated by evolutionary biology. We assign a random fitness value to each vertex of a graph, and a path is said to be \emph{accessible} if the fitness values along it are strictly increasing. \emph{Accessibility percolation} occurs when there exists an accessible path from a distinguished vertex to infinity, or, in the case of a finite graph, to a designated target set of vertices.

Formally, let $G=(V,E)$ be a locally finite, connected graph with a specified root $\rho$. To each vertex $v\in V$, attach a label $X_v$. We consider simple paths starting from $\rho$ along which the labels are increasing; we call such paths \emph{increasing} or \emph{accessible}. The model of \emph{accessibility percolation} on $G$ retains only those vertices $v$ that can be reached from $\rho$ by an accessible path. We say that accessibility percolation occurs if there exists an infinite accessible path, or for finite graphs, an accessible path from $\rho$ to some prescribed target set.

The simplest and most extensively studied setting is the case where the labels $\{X_v\}_{v\neq \rho}$ are independent and identically distributed. This is known as the \emph{House of Cards} model, introduced by Nowak and Krug \cite{nowak2013accessibility}, with a biological motivation \cite{aita2000analysis, duque2023rmf, hegarty2014existence} which we will detail later. On regular trees, this model was studied by Roberts and Zhuo Zhao \cite{roberts_zhao:increasing_paths_regular_trees} and by Chen \cite{chen2014increasing}. Berestycki, Brunet and Shi \cite{berestycki2016nobacksteps, berestycki2017backsteps}, considered the House of Cards model on the discrete hypercube, as well as on an inhomogeneous tree designed to mimic the geometry of the hypercube. Closely related questions have also emerged in the study of \emph{temporal networks}, motivated by computer science, where random labels are typically assigned to the edges rather than the vertices. In that setting, the Erd\H{o}s-R\'enyi graph is the most commonly studied underlying graph; see \cite{angel2020long, broutin2024increasing, casteigts2024sharp}. A recurring feature across all the works mentioned above is that the length of the longest accessible path is of the same order as the degree of a typical vertex.

In this paper, we consider a different model for the labels, called the Rough Mount Fuji (RMF) model, which was also proposed in an evolutionary context as a more realistic fitness landscape than the House of Cards model \cite{aita2000analysis, nowak2013accessibility}.

\begin{definition}[Rough Mount Fuji]
\label{RMF}
     Fix $\theta\in[0,1]$, and let $\{U_v : v\in V\}$ be independent Uniform$(0,1)$ random variables. For each $v\in V$, define
     \begin{equation}\label{eq:RMF}
     X_v= U_v + \theta d(\rho, v)
     \end{equation}
     where $d(\rho,v)$ is the distance from the root $\rho$ to vertex $v$.
\end{definition}

Unless otherwise specified, $d$ will denote the graph distance, though other choices of metric are possible. In particular, later in the paper we consider the case when $d$ is the $\ell^p$ distance on the integer lattice $\Z^n$. One may also restrict the class of admissible paths: for example, \cite{berestycki2016nobacksteps} considers only non-backtracking paths, whereas \cite{berestycki2017backsteps} allows backtracking. On trees, which are the focus of the first part of this paper, all simple paths are non-backtracking and therefore this issue does not arise.

Although other distributions for the random variables $U_v$ in Definition \ref{RMF} can be considered (see, for example, \cite{aita2000analysis}), we restrict attention here to the uniform distribution in order to keep notation and calculations relatively concise. Extending our results to more general distributions is an interesting direction which we intend to consider in future work.

We observe that when $\theta=0$, we recover the House of Cards model. As $\theta$ increases, the deterministic drift term $\theta d(\rho,v)$ makes it progressively easier for labels along paths emanating from the root to be increasing, long increasing paths could exist even on graphs of uniformly bounded degree. For any infinite, locally finite, connected graph $G$, monotonicity implies the existence of a critical value $\theta_c$ such that there is almost surely no infinite accessible path for $\theta<\theta_c$, whereas such paths do exist with strictly positive probability for $\theta>\theta_c$.

Our main contributions are as follows. We give a precise characterisation of the critical parameter $\theta_c$ when $G$ is a Bienaym\'e-Galton-Watson tree, and show that percolation does not occur at the critical point. We also provide explicit bounds on $\theta_c$; in particular, the lower bound holds for any infinite, locally-finite rooted tree. Finally, we show that $\theta_c$ is non-trivial when $G=\mathbb{Z}^n$. When the distance $d$ in \eqref{eq:RMF} is the graph distance, this follows from a coupling with oriented percolation, essentially as in \cite{duque2023rmf, hegarty2014existence}. However, when $d$ is the $\ell^p$ distance for $p>1$, including the Euclidean case $p=2$, a much more delicate coupling argument is required.

\begin{figure}[h]
  \centering
  \begin{minipage}[b]{0.3\textwidth}
    \includegraphics[width=\textwidth]{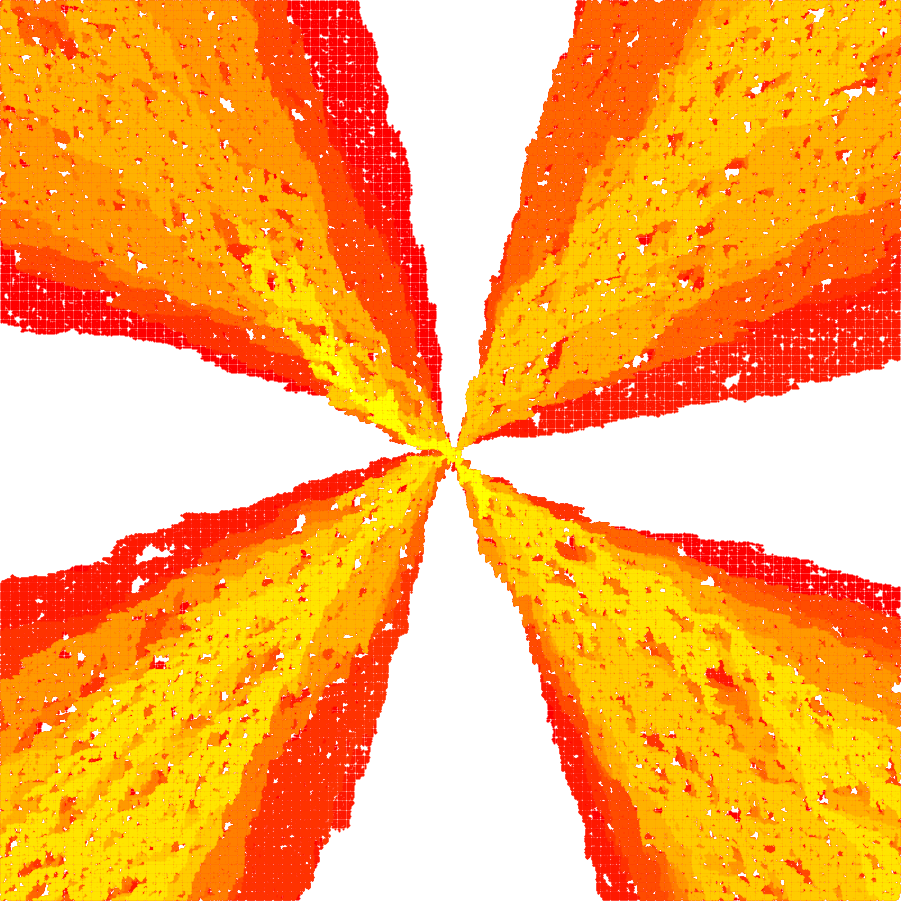}
  \end{minipage}
  \hfill
  \begin{minipage}[b]{0.3\textwidth}
    \includegraphics[width=\textwidth]{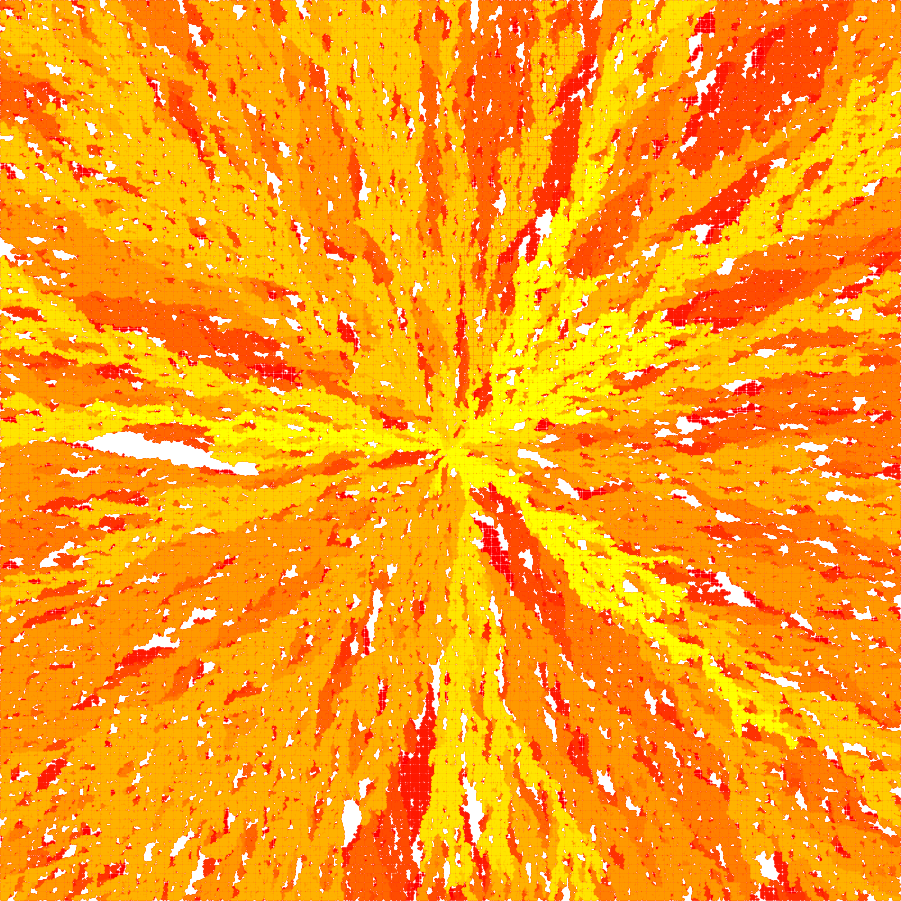}
  \end{minipage}
  \hfill
    \begin{minipage}[b]{0.3\textwidth}
    \includegraphics[width=\textwidth]{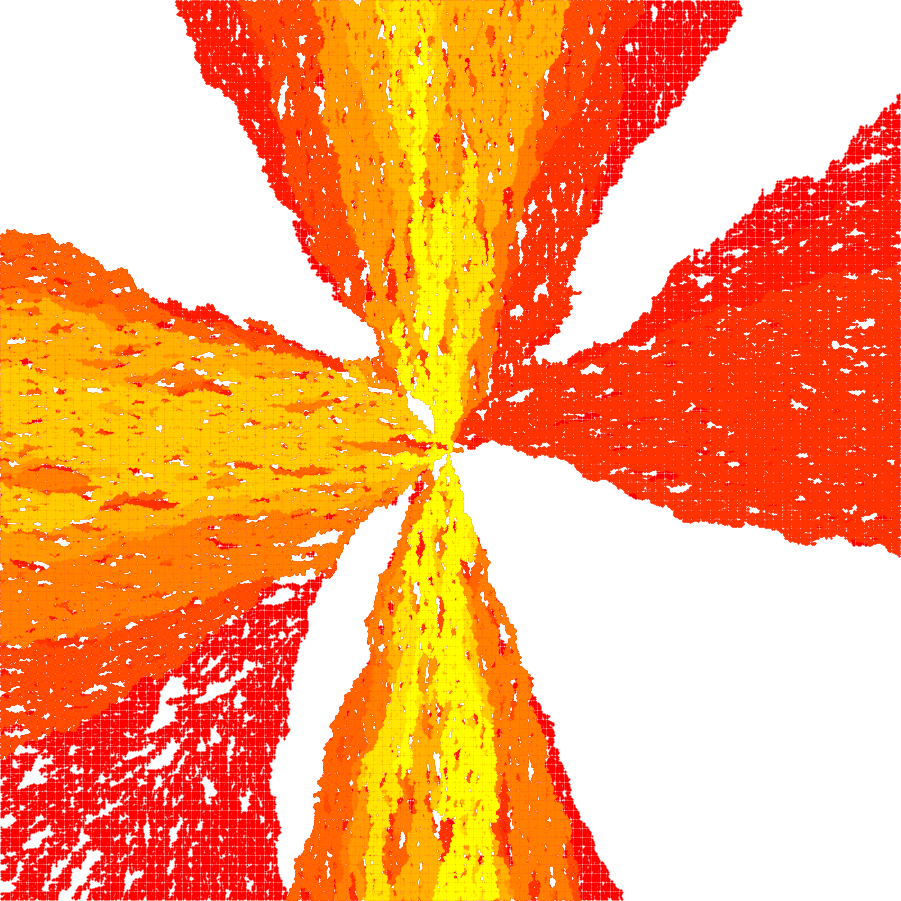}
  \end{minipage}
  \caption{Accessible paths in $\Z^n$ with backsteps. The left-hand image uses the $\ell^1$ distance, and shows several values of $\theta$, from $\theta=0.33$ (yellow) to $0.42$ (red). The middle image uses the $\ell^2$ distance, and shows $\theta=0.45$ (yellow) to $0.54$ (red).The right-hand image uses the $\ell^4$ distance, and shows $\theta=0.48$ (yellow) to $0.57$ (red). With the $\ell^1$ distance, paths near the diagonal are preferred. With the $\ell^4$ distance, paths near the axes are preferred. With the $\ell^2$ distance, there does not appear to be a strong directional preference.}\label{fig:latticebacksteps}
\end{figure}

\begin{figure}[h]
  \centering
  \begin{minipage}[b]{0.3\textwidth}
    \includegraphics[width=\textwidth]{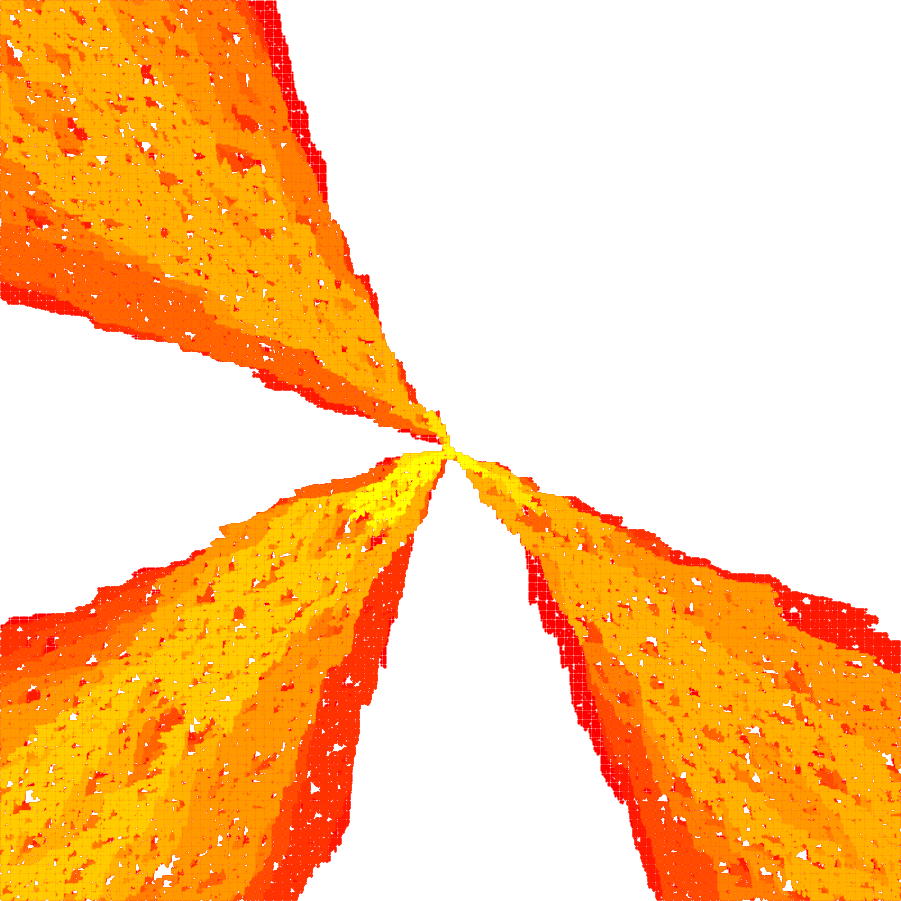}
  \end{minipage}
  \hfill
  \begin{minipage}[b]{0.3\textwidth}
    \includegraphics[width=\textwidth]{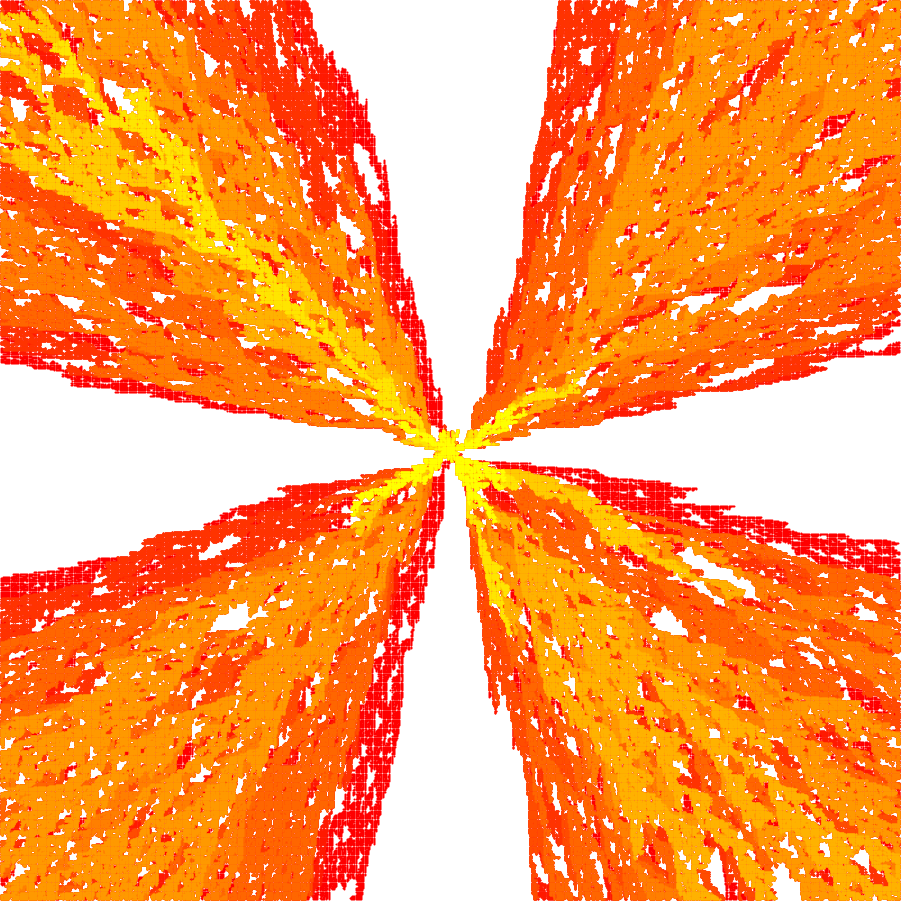}
  \end{minipage}
  \hfill
    \begin{minipage}[b]{0.3\textwidth}
    \includegraphics[width=\textwidth]{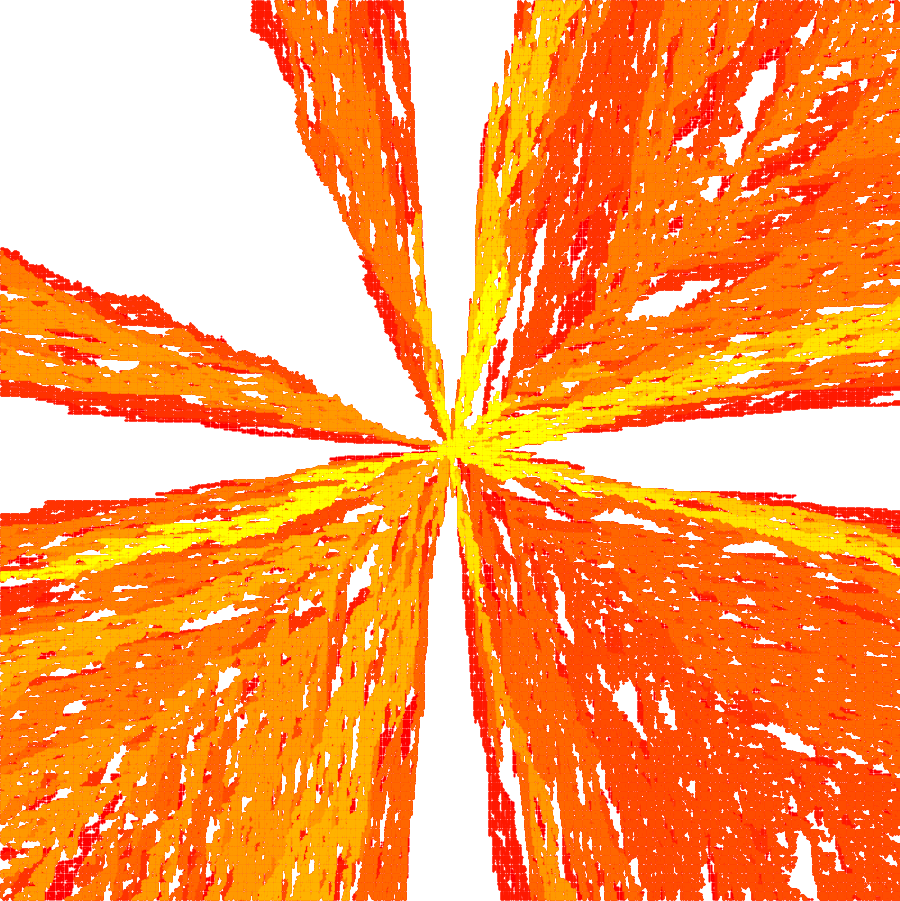}
  \end{minipage}
  \caption{Accessible paths in $\Z^n$ without backsteps. The left-hand image uses the $\ell^1$ distance, and shows several values of $\theta$, from $\theta=0.33$ (yellow) to $0.42$ (red). The middle image uses the $\ell^2$ distance, and shows $\theta=0.45$ (yellow) to $0.54$ (red). The right-hand image uses the $\ell^4$ distance, and shows $\theta=0.53$ (yellow) to $0.62$ (red). In all cases, paths near the axes are inaccessible, but as $q$  
  increases, paths near the diagonal become more difficult too.}\label{fig:latticenobacksteps}
\end{figure}

\subsection{Results on trees}

Consider a rooted, infinite, locally-finite tree $T$, and attach a label $X_v$ to each vertex $v$ according to the Rough Mount Fuji model \eqref{RMF}.

In the case when $T$ is a Bienaym\'e-Galton-Watson tree whose offspring distribution satisfies an $L\log L$ condition, we are able to give an exact characterisation of the critical RMF accessibility percolation threshold.

\begin{theorem}
\label{theomultitype}
For $\theta\in(0,1]$ and $x\in\R$, define
\[Q_\theta(x) = \sum_{j=0}^{\lfloor 1/\theta\rfloor + 1} \frac{(-x)^j}{j!}(1-(j-1)\theta)^j.\]
Let $m_c(\theta)$ be the minimal root of the polynomial equation
\begin{equation}\label{eq:criticalpoly}
    Q_\theta(m) = 0.
\end{equation}
Suppose that $T$ is a Bienaym\'e-Galton-Watson tree whose offspring distribution $L$ satisfies $\E[L]=m$ and $\E[L\log_+ L] < \infty$. There is Rough Mount Fuji accessibility percolation with positive probability on $T$ if and only if $m > m_c(\theta)$.
\end{theorem}

\begin{remark}
In Theorem \ref{theomultitype} we specify a critical parameter $m_c(\theta)$ in terms of $\theta$. However, it is easy to see that $m_c(\theta)$ is strictly decreasing as a function of $\theta$, with $m_c(1)=1$ and $\lim_{\theta\to 0}m_c(\theta)=\infty$ (for the latter, note that $Q_\theta(x)\to e^{-x}$ uniformly on compacts as $\theta\to 0$); and therefore we may also define $\theta_c(m)$ for $m\ge 1$ implicitly as the unique value of $\theta$ such that $m_c(\theta)=m$.
\end{remark}

\noindent
One can see numerically that $m\theta_c(m)\to 1/e$ as $m\to\infty$, or equivalently that $\theta m_c(\theta)\to 1/e$ as $\theta\to 0$. In fact, we are able to give an explicit lower bound on $\theta_c$ that matches this $1/e$ asymptotic, and which holds for much more general trees, not only Bienaym\'e-Galton-Watson trees. In order to state this we must recall the definition of the branching number $\br T$ of a tree $T$.

A detailed discussion of the branching number can be found in \cite{lyons_peres:probability_on_trees}. The formulation that will be most convenient for our purposes requires the concept of a cutset.

\begin{definition}\label{definitionofcutset}
We say that $\Pi\subset V\setminus\{\rho\}$ is a \emph{cutset} separating $\rho$ and $\infty$ if every infinite (simple) path from $\rho$ must include at least one vertex from $\Pi$.    
\end{definition}

In fact the definition of cutset given in \cite{lyons_peres:probability_on_trees} uses edges instead of vertices; but since edges and non-root vertices are in bijection for trees, there is no difference for our purposes.

\begin{definition}
    \label{definitionofbranchingnumber}
Suppose that $G=T$, a rooted infinite, locally finite tree. The \emph{branching number} of $T$ is
\[ \operatorname{br} T=\sup \left\{\lambda : \inf _{\Pi} \sum_{v \in \Pi} \lambda^{-d(\rho,v)}>0\right\}\]
where $d(\rho,v)$ denotes the graph distance from the root $\rho$ to the vertex $v$, and the infimum is taken over cutsets $\Pi$ that separate $\rho$ from infinity.
\end{definition}

This is not the standard definition of branching number, but it is equivalent by the max-flow min-cut theorem \cite{lyons_peres:probability_on_trees}, and this formulation will be the one that is most useful for us. We can now state our lower bound on $\theta_c$ for general trees.

\begin{theorem}\label{thm:no_perco_small_theta}
For any rooted, infinite, locally-finite tree $T$ and any $\theta<\frac{1}{e\operatorname{br}T}$,
\[\P(\text{there exists an infinite increasing path}) = 0.\]
In particular, $\theta_c(T) \ge \frac{1}{e\br T}$.
\end{theorem}

\begin{remark}
    In the case of rooted, infinite, locally finite trees, a simple coupling with Bernoulli percolation shows that there is Rough Mount Fuji accessibility percolation for $\theta > 1/\br T$, i.e.~that $\theta_c(T) \le 1/\br T$. This is implicit in \cite[Theorem 1]{duque2023rmf}.
\end{remark}

It is well-known that on the event of non-extinction, Bienaym\'e-Galton-Watson trees satisfy $\br T = \E[L]$ almost surely (see e.g.~\cite{lyons_peres:probability_on_trees}.) Thus the bound $\theta_c \ge 1/e\br T$ matches the asymptotic $m\theta_c(m)\to e$ mentioned after Theorem \ref{theomultitype}.

We now give two further results for Bienaym\'e-Galton-Watson (BGW) trees, which are less precise than Theorem \ref{theomultitype} but give more explicit bounds that may be easier to apply in practice. The first is a slight improvement on Theorem \ref{thm:no_perco_small_theta} above, showing that for BGW trees there is no accessibility percolation even at $1/em$, and giving an upper bound on the tail of the probability of a long increasing path.

\begin{proposition}\label{prop:no_perco_GW}
Suppose that $T$ is a Bienaym\'e-Galton-Watson tree whose offspring distribution $L$ satisfies $\E[L] = m\in(1,\infty)$. For any $\theta\le \frac{1}{em}$,
\[\P(\text{there exists an infinite increasing path}) = 0,\]
and
\[\limsup_{h\to\infty}\frac{1}{h}\log \P(\text{there exists an increasing path of length $\ge h$ starting from }\rho) \le \log(em\theta).\]
\end{proposition}

Our second bound for BGW trees uses a simple discretisation technique to provide an explicit upper bound on $\theta_c$. One can see from the characterisation of $\theta_c$ in Theorem \ref{theomultitype} that, unlike our lower bound on $\theta_c$, this is not asymptotically tight as $m\to\infty$.

\begin{proposition}\label{prop:perco_GW}
Suppose that $T$ is a Bienaym\'e-Galton-Watson tree whose offspring distribution $L$ satisfies $\E[L]=m$ and $\E[L\log_+ L] < \infty$.
Then for any $\theta > 1-\sqrt{1-\frac{1}{m}}$,
\[\P(\text{there exists an infinite increasing path}) >0.\]
In particular, $\theta_c(m) \le 1-\sqrt{1-\frac{1}{m}} \sim \frac{1}{2 m}$.
\end{proposition}

As an application of our results on trees, we can obtain precise asymptotics for the probability that $h+1$ RMF labels are in order, which may be of independent interest. Note that this result does not involve accessibility percolation or trees; it is a statement only about independent Uniform$(0,1)$ random variables.

\begin{corollary}\label{cor:spectral_prob_est}
There exist constants $0<c_1<1<c_2<\infty$ such that for any $\theta\in(0,1)$,
\[\frac{c_1}{m_c(\theta)^h} \le \P(U_0<U_1+\theta<\ldots<U_h+h\theta) \le \frac{c_2}{m_c(\theta)^h},\]
where $m_c(\theta)$ is the minimal root of the polynomial equation \eqref{eq:criticalpoly}, and in particular
\[\frac{1}{e\theta}\le m_c(\theta) \le \frac{1}{2\theta - \theta^2}.\]
\end{corollary}

\subsection{Results on $\mathbb{Z}^n$}\label{sec:zd_intro}

We now consider the case when $G$, our underlying graph, is the integer lattice $\mathbb{Z}^n$ for $n\ge 2$. In this case, we will often write $0$ rather than $\rho$ for the root (origin) of the graph. Moreover, in this setting, it is important to specify:
\begin{enumerate}[(a)]
\item which paths are permitted: only those moving monotonically away from $0$, or including those that are allowed to ``backtrack'', i.e.~sometimes take a step towards $0$;
\item which metric $d$ is used in Definition \ref{RMF}.
\end{enumerate}
These choices are essentially trivial in the case of trees: every (simple) path is non-backtracking, and the graph distance is the natural metric. But each choice has material consequences for the model on $\mathbb{Z}^n$.

For paths, we consider both options mentioned above:
\begin{enumerate}
    \item[(a.i)] the \emph{non-backtracking} case, where we restrict to paths where each step moves further from the origin, i.e.~simple paths $x_0,x_1,\ldots$ where $x_0=0$ and $d(0,x_{j+1})>d(0,x_j)$ for each $j$;
    \item[(a.ii)] all (simple) paths: there is no restriction.
\end{enumerate}
For the metric $d$, the graph distance is the simplest choice, and we give a proof of that case first, essentially using existing methods. However, it is also natural to consider the Euclidean distance, and this leads to qualitative differences in the set of accessible paths: see Figures \ref{fig:latticebacksteps} and \ref{fig:latticenobacksteps}. Our main proof will allow us to use the $\ell^q$ distance $d(x,y) = \|x-y\|_q$ for any $q\in(1,\infty]$ (we use $q$ rather than $p$ to avoid confusion with another parameter $p$ that we will introduce shortly).

Recall that we say that rough Mount Fuji accessibility percolation occurs if there is an infinite path from $0$ along which the labels are increasing.  In each case above we define
\[\theta_c = \inf\{\theta\in[0,1] : \P_\theta(\text{RMF accessibility percolation})>0\},\]
noting that of course the value of $\theta_c$ depends on the distance $d$, the dimension $n$, and the choice of non-backtracking or all paths.

In fact, in the case that we allow backtracking paths, the system is not monotone in $\theta$; we demonstrate with a caricature of a realisation in Figure \ref{fig:counterexample}, where for $\theta<0.2$ and $\theta>0.4$, there are potentially long accessible paths, whereas for $\theta\in(0.2,0.4)$ there are not. (Of course, in reality the Uniform$(0,1)$ random variables will not be multiples of $0.1$, but a small perturbation of this realisation will occur with positive probability and still demonstrate non-monotonicity in $\theta$.) Thus, to cover the case that we allow backtracking paths, we should also define a second critical value of $\theta$,
\[\theta'_c = \sup\{\theta\in[0,1] : \P_\theta(\text{RMF accessibility percolation})=0\}.\]

\begin{figure}[h]
  \centering
  \includegraphics[width=0.5\textwidth]{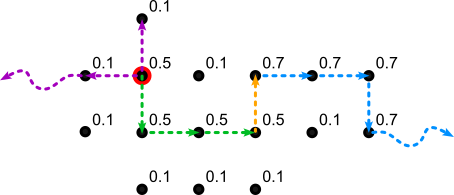}
  \caption{A caricature of a realisation on $\mathbb{Z}^2$, using the graph distance and allowing backtracking paths, where we do not have monotonicity in $\theta$. The circled red vertex is the origin. The numbers by the vertices show the independent Uniform$(0,1)$ random variables. For $\theta<0.2$, the orange edge is accessible and there may be a long accessible path continuing along the blue arrow to the right. For $\theta\in(0.2,0.4)$, the orange edge is no longer accessible, and the longest accessible path is marked in green. For $\theta>0.4$ there may be long accessible paths e.g.~following the purple arrow to the left.}\label{fig:counterexample}
\end{figure}

When we do not allow backtracking paths, by monotonicity, we have that rough Mount Fuji accessibility percolation occurs with strictly positive probability for $\theta>\theta_c$, and with zero probability for $\theta<\theta_c$; in particular $\theta'_c = \theta_c$.

When $\theta = 0$, it is easy to show that the probability of accessibility percolation is zero; and when $\theta=1$, the path along each axis (and in fact every non-backtracking path in the case when $d$ is the graph distance) is guaranteed to be increasing. Our main aim is to show that the critical values for accessibility percolation are non-trivial in all cases.

\begin{theorem}\label{thm:non_trivial_phase_transition}
    For any $n\ge 2$, in each of the cases (a.i) and (a.ii) above, when $d$ is the $\ell^q$ distance for any $q\in[1,\infty]$, the Rough Mount Fuji accessibility percolation critical values are non-trivial, i.e.~$0<\theta_c\le \theta'_c <1$.
\end{theorem}

We break this result down into various upper and lower bounds. First, we can give an easy lower bound on $\theta_c$ by embedding a tree in the lattice. In fact, our proof works for very general graphs.

\begin{proposition}\label{prop:lattice_no_perc}
	Suppose that $\mathcal{G}$ is a rooted, bounded degree graph, and that $d$ is a metric on the vertices of $\mathcal{G}$ such that $d(u,v)\le 1$ for any neighbouring vertices $u$ and $v$. Let $\Delta$ be the maximum degree of $\mathcal{G}$. If $\theta \le \frac{1}{(\Delta-1)e}$, then there is no Rough Mount Fuji accessibility percolation on $\mathcal{G}$.
	
    In particular, for any $n\ge2$ and any $q\in[1,\infty]$, there is no Rough Mount Fuji accessibility percolation on the $n$-dimensional lattice with the $\ell^q$ distance
    \begin{enumerate}[(i)]
        \item with non-backtracking paths when $\theta \leq \frac{1}{ne}$;
        \item with all paths when $\theta \le \frac{1}{(2n-1)e}$.
    \end{enumerate}
    Thus we have $\theta_c \ge 1/(ne)$ in the non-backtracking case, and $\theta_c \ge 1/((2n-1)e)$ otherwise.
\end{proposition}

\begin{remark}
In fact the proof of Proposition \ref{prop:lattice_no_perc} works for any \emph{directed} graph with out-degree bounded by $\Delta-1$. This is what we will use to prove part (i) above.
\end{remark}

We now turn to the more difficult task of showing that $\theta'_c<1$. Clearly $\theta'_c$ is decreasing in the dimension $n$, and clearly if $\theta'_c<1$ when considering only non-backtracking paths, then $\theta'_c<1$ when allowing all simple paths. Thus it suffices to consider the case of non-backtracking paths in $\mathbb{Z}^2$.

As mentioned above, the simplest choice is the graph distance $d(x,y) = \|x-y\|_1$, in which case we can bound $\theta_c$ from above via a straightforward coupling with oriented site percolation; see below for a brief introduction to oriented percolation. The proof of this result is essentially the same as in \cite{duque2023rmf} and \cite{hegarty2014existence}.

\begin{proposition}\label{prop:lattice_perc_l1}
For any $n\ge 2$, there is Rough Mount Fuji accessibility percolation with positive probability on $\mathbb{Z}^n$ with the $\ell^1$ (graph) distance, when $\theta>p^\circ_c$, where $p^\circ_c < 1$ is the critical parameter for oriented site percolation on $\mathbb{Z}^2$. Thus $\theta'_c \le p^\circ_c$. This holds in the case of non-backtracking paths and therefore also when considering all simple paths.
\end{proposition}

Oriented (site) percolation is a variation of Bernoulli percolation where each edge is assigned a direction, usually specified to be away from the origin. For example, in two dimensions, for each $x,y\in \mathbb{Z}_+$, we consider the directed edges from $(x,y)$ to $(x+1,y)$ and to $(x,y+1)$. We then let each site be open independently with probability $p\in[0,1]$ and ask for the probability that there exists an infinite open path of directed edges starting from $0$. (We can also define oriented bond percolation by opening each directed edge independently with probability $p$, rather than each site.) As for standard Bernoulli percolation, it is known that there is a non-trivial critical value $p^\circ_c$ such that for $p>p^\circ_c$ there is positive probability of an infinite open directed path starting from $0$, whereas for $p<p^\circ_c$, no infinite open directed paths exist almost surely. In two dimensions, it is known \cite{balister1993upper} that $p^\circ_c \le 0.72599$. We do not expect $\theta_c$ (or $\theta'_c$) to be near $p^\circ_c$ in practice. See \cite{duque2023rmf} for simulations suggesting that for non-backtracking paths with the graph distance, $\theta_c\approx 0.33$.

As indicated above, while a straightforward coupling with oriented site percolation allows us to prove non-triviality of the phase transition when using the $\ell^1$ (graph) distance, if we wish to use the $\ell^q$ distance for $q>1$, we require further work. This is due to the fact that neighbouring vertices can have arbitrarily small difference in their distances to the origin: for example $\|(x,0)\|_q = x$ and $\|(x,1)\|_q \approx x + x^{1-q}/q = x + o(1)$ when $x$ is large. This means that, recalling \eqref{eq:RMF}, the labels when moving from one vertex to a neighbour do not increase by $\theta$ on average, but rather by an amount that depends on their position in the graph. This breaks the coupling used to prove Proposition \ref{prop:lattice_perc_l1}. Nonetheless we are able to show non-triviality in this more difficult setting. We give a brief sketch of the proof below.

\begin{proposition}\label{prop:lattice_perc_l2}
    For any $q\in(1,+\infty]$, there is Rough Mount Fuji accessibility percolation with positive probability on $\mathbb{Z}^2$ with non-backtracking paths and the $\ell^q$ distance for all sufficiently large $\theta<1$. In particular $\theta'_c<1$.
\end{proposition}

Since this proposition holds in the case of non-backtracking paths on $\mathbb{Z}^2$, we therefore also have $\theta'_c<1$ on $\mathbb{Z}^n$ for any $n\ge 2$, when considering either non-backtracking or all simple paths. Theorem \ref{thm:non_trivial_phase_transition} therefore follows immediately from Propositions \ref{prop:lattice_no_perc}, \ref{prop:lattice_perc_l1} and \ref{prop:lattice_perc_l2}.

As mentioned above, the coupling with oriented percolation used to prove Proposition \ref{prop:lattice_perc_l1} cannot work for Proposition \ref{prop:lattice_perc_l2} since distances do not increase by $1$ with each step along a path; in fact in the $q=+\infty$ case, there are many edges with no increase in distance at all. Our strategy to avoid this issue is somewhat counterintuitive. As we saw above, for large (but finite) $q$, arbitrarily small distances between neighbours occur close to the axes. Indeed, when we disallow backtracking, it is simply not possible to see accessible paths near the axes; see Figure \ref{fig:latticenobacksteps}. However, from a theoretical standpoint, it is not even clear that there are infinite accessible paths near the diagonal when $\theta=1$, whereas along each axis, we clearly do have at least one accessible path when $\theta=1$, and it is this property that draws us, in the proof, to restrict our class of paths to those remaining within a cone near the horizontal axis. This preference for paths near the axes when $q$ is large can also be seen in simulations when backtracking is allowed---see Figure \ref{fig:latticebacksteps}---and certainly the difficulty of following paths near the diagonal can also be observed when backtracking is not allowed---see Figure \ref{fig:latticenobacksteps}.

By choosing $\theta$ sufficiently close to $1$, we can then construct a coupling with a Bernoulli percolation model where all \emph{horizontal edges} are open with probability arbitrarily close to $1$. This is not true for vertical edges, but for each vertical edge in the north-east quadrant, there is probability at least $1/2$ of the label of the vertex at the top of the edge being larger than the label of the vertex at the bottom of the edge. Thus we construct an inhomogeneous edge percolation model; but even then, the vertical edges are not independent. To circumvent this final difficulty we use a two-scale method, creating long, thin ``bricks'' which we show contain suitable open paths with high probability.

To summarise the proof strategy,
\begin{itemize}
\item we consider a coupling with oriented percolation, in line with Proposition \ref{prop:lattice_perc_l1};
\item we restrict to a narrow cone near the horizontal axis, so that horizontal edges have probability close to $1$ of being open and vertical edges have probability at least $1/2$ of being open;
\item and we use a two-scale method to manage the inhomogeneity and the dependence of vertical edges.
\end{itemize}

We do not expect either the lower bounds given by Proposition \ref{prop:lattice_no_perc} or the upper bounds given by Propositions \ref{prop:lattice_perc_l1} or \ref{prop:lattice_perc_l2} to be close to the true critical RMF accessibility percolation threshold. Simulations suggest that for non-backtracking paths with the graph distance, $\theta_c\approx 0.33$: see \cite{duque2023rmf} and Figure \ref{fig:latticenobacksteps}.

\subsection{Motivation}

Accessibility percolation was introduced by Nowak and Krug \cite{nowak2013accessibility} as a model for evolution where potential genotypes or phenotypes are represented as the vertices of a graph. Each vertex has an associated fitness, corresponding to the ability of a living organism to survive and reproduce. Mutations in the species' DNA correspond to moving along an edge to a neighbouring vertex; and on large time scales, it is assumed that only species whose fitnesses are greater than their predecessors can fixate in the population. The focus therefore lies on identifying paths of vertices whose fitness values progressively increase; these paths can be seen as viable evolutionary trajectories.

Different biological hypotheses then correspond to different \emph{fitness landscapes}, or distributions for the collection of random variables that capture the fitnesses of the genotypes or phenotypes. The simplest possibility is to assume that the fitnesses are independent and identically distributed. This is called the House of Cards (HoC) model \cite{kingman1978simple}, and most mathematical works on accessibility percolation have used this assumption due to its tractability. From the biological perspective, the HoC model implies that the effect of each and every mutation on fitness is so significant that it totally disrupts the underlying genotype’s fitness architecture, analogous to the collapse of a house of cards.

In this paper, we consider a model where fitness values are not i.i.d.~but are influenced by both new mutations and preceding genotypes; there is a deterministic bias or evolution-oriented drift over time/generation. This is called the Rough Mount Fuji model, reflecting the fact that fitnesses ascend at a constant rate on average over time/generation, but with significant noise around this gradient. This model was also considered by Nowak and Krug \cite{nowak2013accessibility} but has received less attention, especially in the mathematical literature, due to the more complex fitness landscape.

\subsection{Existing and related work}

As mentioned above, accessibility percolation was first considered by Nowak and Krug in \cite{nowak2013accessibility}. They gave bounds on the length of increasing paths in the house of cards (HoC) model on regular $n$-ary trees, in particular showing that the probability of seeing a path of length $h$ converges to $0$ when $n,h\to\infty$ provided that $n/h\le 1/e$. They also identified the accessibility percolation threshold for a specific case of the rough Mount Fuji model with non-uniform noise, again on regular trees.

Roberts and Zhao \cite{roberts_zhao:increasing_paths_regular_trees} improved on the result of Nowak and Krug for the HoC model. They again considered a regular $n$-ary tree, for which every vertex except the root is labelled with an independent and identically distributed continuous random variable. They showed that if $n = \lfloor \alpha/ h\rfloor$ for some fixed $\alpha > 1 / e$, then the probability that there exists an increasing path of length $h$ from the root converges to 1 as $h \rightarrow \infty$, complementing the result of \cite{nowak2013accessibility} that the probability converges to 0 if $\alpha \le 1 / e$. Chen \cite{chen2014increasing} went further and identified the critical length of increasing paths in the $n$-ary tree to be $en-\frac{3}{2} \log n$.

Hegarty and Martinsson \cite{hegarty2014existence} analysed accessibility percolation on the unit hypercube of dimension $n$, considering non-backtracking paths only, and allowing fairly general distributions for the random variables used to generate the fitnesses. They considered both HoC and RMF models; in fact, a version of the HoC model in which the all-ones vertex has fixed fitness 1, and the all-zeroes vertex is given either a fixed fitness or an i.i.d.~fitness. They showed that as $n$ tends to infinity, the probability of finding at least one accessible path from the all-zeroes vertex to the all-ones vertex converges to 0 for their HoC model with random fitness at $0$; and that the probability goes sharply from $0$ to $1$ as the fixed fitness at $0$ passes $\frac{\log n}{n}$. They also showed that the probability of finding an accessible path converges to $1$ for the RMF model (where the all-zeroes and all-ones vertices both also have random fitnesses) provided $\theta=\theta(n)\gg 1/n$.

Duque et al.~\cite{duque2023rmf} also investigated RMF accessibility percolation with non-backtracking paths. They proved a slight generalisation of the Hegarty and Martinsson \cite{hegarty2014existence} result on the hypercube, allowing for even more general fitness distributions. They also used a coupling with oriented percolation to give an upper bound on the RMF accessibility percolation threshold $\theta_c$ on the square lattice and another related lattice, both using the graph distance. They also conducted simulations to estimate the percolation threshold for RMF in 2-ary and 3-ary trees, the square lattice and the related lattice, suggesting that the critical values are in the intervals $[0.2,0.22]$, $[0.12,0.14]$, $[0.3,0.33]$ and $[0.19,0.22]$ respectively.

Berestycki, Brunet and Shi \cite{berestycki2016nobacksteps, berestycki2017backsteps} carried out a more in-depth investigation of HoC accessibility on the hypercube, improving on the methods of Hegarty and Martinsson \cite{hegarty2014existence}. In \cite{berestycki2016nobacksteps} they showed that when considering non-backtracking paths only, if the fitness at $0$ is set to equal $x/n$, then the number of accessible paths from $0$ to $1$ rescaled by $n$ converges in distribution to $e^{-x}$ times the product of two independent exponential random variables. In \cite{berestycki2017backsteps} they allowed paths to backtrack and gave bounds on the probability that an accessible path exists.

The question of finding the longest increasing path, when independent random variables are assigned to each edge of a graph, has also been studied in computer science, in the context of \emph{temporal networks} \cite{angel2020long, broutin2024increasing, casteigts2024sharp}. The idea is that we have some underlying graph whose edges are opened at random times, the key question being which vertices can be reached by following edges that appear in increasing order. Several papers have considered this problem on the Erd\H{o}s-R\'enyi graph $G(n,p)$, often concentrating on the case $p\asymp\frac{\log n}{n}$ which turns out to be the threshold for connectivity of the (accessibility percolated) graph \cite{broutin2024increasing, casteigts2024sharp}. Angel \emph{et al.}~\cite{angel2020long}~gave an asymptotically tight characterisation of the length of the longest increasing path for Erd\H{o}s-R\'enyi graphs when the edge probability $p$ is small; they used a very similar argument to that in \cite{roberts_zhao:increasing_paths_regular_trees} to bound the probability of long accessible paths in trees.

\subsection{Open problems and future work}

There is no reason why the noise in \ref{eq:RMF} needs to consist of independent and identically distributed Uniform$(0,1)$ random variables. Indeed, as mentioned above, Nowak and Krug \cite{nowak2013accessibility} described a setting where the probability $\P(U_0<U_1+\theta<\ldots<U_h+h\theta)$ that a single path is increasing can be calculated exactly. It may be interesting to consider whether our proofs can be extended to cover a wider class of distributions.

Corollary \ref{cor:spectral_prob_est} gave explicit bounds on $m_c(\theta)$ in the tree case, and thus on $\P(U_0<U_1+\theta<\ldots<U_h+h\theta)$. However, there is a gap between the upper and lower bounds, and in particular we expect that the lower bound is tight as $\theta\to 0$. Can one obtain better explicit bounds, perhaps improving the upper bound by using a more delicate argument than the symmetrization in the proof of Proposition \ref{prop:perco_GW}?

When we allow backtracking paths of $\mathbb{Z}^d$, the system is not monotone in $\theta$, as highlighted in Figure \ref{fig:counterexample}. However, it seems natural to expect the probability of accessibility percolation (i.e.~of infinite increasing paths) to be monotone, and thus to expect $\theta_c=\theta'_c$. Is this the case? Moreover, can one find better bounds on $\theta_c$ and/or $\theta'_c$?

Finally, we plan to investigate RMF accessibility percolation on other graphs, specifically the hypercube and Erd\H{o}s-R\'enyi graphs. As mentioned above, Hegarty and Martinsson \cite{hegarty2014existence} showed that the probability of finding an accessible path from $0$ to $1$ converges to $1$ provided that $\theta=\theta(n)\gg 1/n$. We expect to improve on this by showing that the critical $\theta$ is of order $1/n$. In the Erd\H{o}s-R\'enyi case, we are not aware of any work with RMF labels. Most work has concentrated on the House of Cards model on $G(n,p)$ where labels are on the edges and $p$ is of order $\frac{\log n}{n}$. In the RMF case we expect interesting behaviour on $G(n,p)$ when $p$ crosses the ``double jump'' critical point $p=1/n$.

\subsection{Layout of article}

We begin, in Section \ref{sec:onepath}, with an explicit and direct upper bound on the probability that one path of length $h+1$ has increasing RMF labels. In Section \ref{sec:explicit_bounds} we then apply this with a simple first moment bound in Section  to show that there is no accessibility percolation on trees when $\theta$ is sufficiently small, proving Theorem \ref{thm:no_perco_small_theta} and Proposition \ref{prop:no_perco_GW}. In the same section, we also apply a multitype branching process argument with a discretisation of the vertex labels to show that there is accessibility percolation on Bienaym\'e-Galton-Watson trees when $\theta$ is sufficiently large, proving Proposition \ref{prop:perco_GW}.

In Section \ref{sec:characterisation} we use a martingale approach to give an exact characterisation of the critical value on Bienaym\'e-Galton-Watson trees, proving Theorem \ref{theomultitype}. Then in Section 5, we move on to investigating the lattice $\mathbb{Z}^n$, proving Theorem \ref{thm:non_trivial_phase_transition}. We first carry out the straightforward parts of our lattice proofs, showing that there is no accessibility percolation for small $\theta>0$, and demonstrating existence of accessibility percolation for $\theta$ sufficiently close to $1$ in the easier $\ell^1$ case. Finally, in Section \ref{sec:bricklayer}, we introduce the most novel part of our paper: the two-scale coupling that allows us to show accessibility percolation for $\theta$ sufficiently close to $1$ for the $\ell^q$ distance with $q>1$.

\section{Upper bound on the probability of one path being increasing}\label{sec:onepath}
In this section, we will provide an upper bound on the probability that the labels along one single path are increasing. We can then use this probability in a first moment method to prove that there is no Rough Mount Fuji accessibility percolation on general trees when $\theta$ is sufficiently small. This will give us a non-trivial lower bound on the location of the phase transition for RMF accessibility percolation on trees.

For this section, let $G$ be a semi-infinite path consisting of vertices $0,1,2,3,\ldots$ (with the root $\rho=0$) and let $X_j$ be the label attached to vertex $j$, so that, in line with \eqref{RMF},
\[X_j= U_j + \theta j,\]
where $\{U_j: j\ge 0\}$ are independent and identically distributed Uniform$(0,1)$ random variables and $\theta\in [0,1]$. For a path to be accessible up to height $h$, we need that $X_0<X_1< X_2<\ldots < X_h$, which is equivalent to $U_0 < U_1+\theta < U_2+ 2\theta < \ldots < U_h + h\theta$. Our main aim in this section is to prove an upper bound on the probability of this event.

\begin{proposition}\label{prop:path_upper_bd}
    For any $h\in\mathbb{N}$ and $\theta\in[0,1]$,
    \[\mathbb{P}\left(U_0< U_1+ \theta < U_2 + 2\theta<\ldots < U_h +h\theta\right) \le \frac{(1+\theta h)^{h+1}}{(h+1)!}.\]
\end{proposition}

Our strategy is to consider two possibilities for each of the labels: for some values of $j$, the uniform random variables will be ``in order'', by which we mean that $U_j<U_{j+1}$; whereas for some, the uniforms will be ``out of order'', i.e.~$U_j\ge U_{j+1}$, and yet still $U_j < U_{j+1}+\theta$.

More precisely, setting $i_0=-1$ and specifying $i_1,\ldots,i_n\in\{0,\ldots,h-1\}$ to represent the ``out of order'' uniforms, we would like to estimate
\[\mathbb{P}\left(\bigcap\limits_{k=1}^{n}\{U_{i_k}\in [U_{i_k+1},U_{i_k+1}+\theta)\}\cap \bigcap\limits_{\ell=1}^{n}\{U_{i_{\ell-1}+1}<\ldots <U_{i_\ell}\} \cap \{ U_{i_n+1}<\ldots<U_h\}\right).\]
Note that if $i=j$, then we use the convention that $\{U_i<\ldots <U_j\}=\Omega$ where $\Omega$ is the full sample space.

To labour the point, this event is saying that $U_0<\ldots<U_{i_1}$, i.e.~the first $i_1$ uniforms are in order; then $U_{i_1}\in [U_{i_1+1},U_{i_1+1}+\theta)$, i.e.~$U_{i_1}$ is then larger than, but within $\theta$ of, the following uniform. Then the following uniforms are in order again until $U_{i_2}$, which is larger than but within $\theta$ of $U_{i_2+1}$. And so on.  

In order to estimate this probability from above, we will work by induction on $n$, ``peeling off'' the last ``out of order'' uniform by integrating out the values of $U_{i_n}$ and $U_{i_n+1}$ to obtain a probability involving only $n-1$ ``out of order'' uniforms. The following lemma essentially does this, although the statement looks somewhat abstract.

For $i_0=-1$ and specified $i_1<\ldots<i_n\in\{0,\ldots,h-1\}$, let
\[A_k := \{U_{i_{k-1}+1}<\ldots <U_{i_k} \}\]
be the event that the uniforms between $i_{k-1}$ and $i_k$ are in increasing order, and let
\[B_k := \{U_{i_k}\in [U_{i_k+1},U_{i_k+1}+\theta)\}\]
be the event that the $i_k$th uniform is greater than but within $\theta$ of the following uniform.

\begin{lemma}
\label{lemmapreinductionnegative}
    Let $n \in \mathbb{N}$, $h\geq n+1$ and $c \in[0,\infty)$. Let $i_0=-1$ and $i_1<\ldots<i_n\in\{0,\ldots,h-1\}$. Define
\[p_n^{(h)}(c):= \int\limits_0^1\int\limits_0^1 \frac{(1-v_{n}+c)^{h-i_{n}-1}}{(h-i_{n}-1)!}\, \mathbb{P}\left(\bigcap_{k=1}^n A_k \cap \bigcap_{\ell=1}^n B_\ell \cap \{U_{i_{n}+1} \in dv_n,\, U_{i_n} \in du_n\}\right).\]
    Then for any $n\ge 2$,
    \[p_n^{(h)}(c) \leq p_{n-1}^{(h)}(c + \theta)- p_{n-1}^{(h)}(c).\]
\end{lemma}
 
Before proving Lemma \ref{lemmapreinductionnegative} we state a more elementary lemma that will be needed in the proof.

\begin{lemma}
\label{lem:integrals}
For any $\theta\ge 0$, $c\in[0,\infty)$, $u\in\mathbb{R}$ and $a \in \mathbb{N}$,
    \begin{equation}\label{integral1}
        \int\limits_{u-\theta}^{u} \frac{(1-v+c)^a}{a!} dv =\frac{(1-u+\theta+c)^{a+1}}{(a+1)!} -\frac{(1-u+c)^{a+1}}{(a+1)!}
    \end{equation}
and for any $a, b \in \mathbb{N}$, $\theta>0$ and $v\in[0,1]$,
    \begin{equation}\label{integral2}
        \int\limits_{v}^{1}\frac{(w-v)^a}{a!}\frac{(1-w+c)^b}{b!} dw \leq \frac{(1-v+c)^{a+b+1}}{(a+b+1)!}.
    \end{equation}
For any $a, b \in \mathbb{N}$ and $v\in[0,1]$,
    \begin{equation}\label{integral3}
        \int\limits_{v}^{1}\frac{(w-v)^a}{a!}\left[\frac{(1-w+\theta+c)^b}{b!}-\frac{(1-w+c)^b}{b!}\right] dw \leq \frac{(1-v+\theta+c)^{a+b+1}}{(a+b+1)!}-\frac{(1-v+c)^{a+b+1}}{(a+b+1)!}.
        \end{equation}
\end{lemma}

\begin{proof}
We note that \eqref{integral1} is trivial, and prove \eqref{integral2} using integration by parts. We have
\[\int\limits_{v}^{1}\frac{(w-v)^a}{a!}\frac{(1-w+c)^b}{b!} dw = \frac{c^b(1-v)^{a+1}}{b!(a+1)!}  + \int\limits_v^1 \frac{(w-v)^{a+1}(1-w+c)^{b-1}}{(a+1)!(b-1)!} dw,\]
and a simple recursion then shows that
\[\int\limits_{v}^{1}\frac{(w-v)^a}{a!}\frac{(1-w+c)^b}{b!} dw = \frac{1}{(a+b+1)!} \sum_{k=a+1}^{a+b+1} {{a+b+1}\choose k} c^{a+b+1-k}(1-v)^k.\]
Bounding the sum from above by including the terms $k=0$ to $k=a$, we see that
\begin{align*}
    \int\limits_{v}^{1}\frac{(w-v)^a}{a!}\frac{(1-w+c)^b}{b!} dw &\le \frac{1}{(a+b+1)!} \sum_{k=0}^{a+b+1} {{a+b+1}\choose k} c^{a+b+1-k}(1-v)^k\\
    &=\frac{(1-v+c)^{a+b+1}}{(a+b+1)!}
\end{align*}
as claimed. To prove \eqref{integral3}, we use \eqref{integral1} and \eqref{integral2}.
    \begin{align*}
        &\int\limits_{v}^{1}\frac{(w-v)^a}{a!}\left[\frac{(1-w+\theta+c)^b}{b!}-\frac{(1-w+c)^b}{b!}\right] dw \\
        &= \frac{1}{(a+b+1)!} \sum_{k=0}^{a+b+1} {{a+b+1}\choose k} (1-v)^k\left[(\theta+c)^{a+b+1-k}-c^{a+b+1-k}\right]\\
        & \, - \frac{1}{(a+b+1)!} \sum_{k=0}^{a} {{a+b+1}\choose k}(1-v)^k \left[ (\theta+c)^{a+b+1-k}-c^{a+b+1-k} \right]\\
        &\leq \frac{(1-v+\theta+c)^{a+b+1}}{(a+b+1)!}-\frac{(1-v+c)^{a+b+1}}{(a+b+1)!}.
    \end{align*}
    as claimed.
\end{proof}

We now apply this to prove Lemma \ref{lemmapreinductionnegative}.

\begin{proof}[Proof of Lemma \ref{lemmapreinductionnegative}.]
   First, suppose that $i_{n-1}\neq i_n-1$; the other case is slightly simpler and we carry it out separately. First note that 
\begin{align*}
     & p_n^{(h)}(c) = \int\limits_0^1\int\limits_{u_n-\theta}^{u_n} \mathbb{P}\left(\bigcap\limits_{k=1}^{n-1}A_k \cap \bigcap\limits_{\ell=1}^{n-1} B_\ell \cap \{U_{i_{n-1}+1}<\ldots <U_{i_{n}-1}<u_n\}\right)\\
     &\hspace{80mm}\cdot \frac{(1-v_n+c)^{h-i_{n}-1}}{(h-i_{n}-1)!} d v_{n}d u_{n}.
\end{align*}
    By \eqref{integral1}, this is equal to
\begin{align*}
     &\int\limits_0^1 \mathbb{P}\left(\bigcap\limits_{k=1}^{n-1}A_k \cap \bigcap\limits_{\ell=1}^{n-1} B_\ell \cap \{U_{i_{n-1}+1}<\ldots <U_{i_{n}-1}<u_n\}\right)\\
     &\hspace{50mm}\cdot\left[\frac{(1-u_n+\theta+c)^{h-i_n}}{(h-i_n)!} -\frac{(1-u_n+c)^{h-i_n}}{(h-i_n)!}\right]du_n\\
     &= \int\limits_0^1\int\limits_{v_{n-1}}^1 \mathbb{P}\left(\bigcap\limits_{k=1}^{n-1}A_k \cap \bigcap\limits_{\ell=1}^{n-2}B_\ell \cap \{U_{i_{n-1}}\in (v_{n-1}, v_{n-1}+\theta)\}\right)\\
     &\hspace{22mm} \cdot \frac{(u_n-v_{n-1})^{i_n-i_{n-1}-2}}{(i_n-i_{n-1}-2)!}\left[\frac{(1-u_n+\theta +c)^{h-i_{n}}}{(h-i_{n})!}-\frac{(1-u_n+c)^{h-i_{n}}}{(h-i_{n})!}\right] d u_{n} d v_{n-1}.
\end{align*}
    Then by \eqref{integral2}, this is at most
\begin{align*}
     &\int\limits_0^1 \mathbb{P}\left(\bigcap\limits_{k=1}^{n-2}A_k \cap \bigcap\limits_{\ell=1}^{n-2} B_\ell \cap \{U_{i_{n-1}}\in (v_{n-1}, v_{n-1}+\theta)\} \cap\{U_{i_{n-2}+1}<\ldots <U_{i_{n-1}-1}<u_{n-1}\}\right)\\
     &\hspace{35mm} \cdot \left[\frac{(1-v_{n-1}+\theta+c)^{h-i_{n-1}-1}}{(h-i_{n-1}-1)!}-\frac{(1-v_{n-1}+c)^{h-i_{n-1}-1}}{(h-i_{n-1}-1)!}\right] d v_{n-1}\\
&=\int\limits_0^1\int\limits_0^1 \left[\frac{(1-v_{n-1}+c+\theta)^{h-i_{n-1}-1}}{(h-i_{n-1}-1)!}-\frac{(1-v_{n-1}+c)^{h-i_{n-1}-1}}{(h-i_{n-1}-1)!}\right] \\
&\hspace{35mm}\cdot \mathbb{P}\bigg(\bigcap\limits_{k=1}^{n-1}A_k \cap \bigcap\limits_{\ell=1}^{n-1} B_\ell \cap \{U_{i_{n-1}+1} \in dv_{n-1}, U_{i_{n-1}} \in du_{n-1}\}\bigg) \\
    & = p_{n-1}^{(h)}(c +\theta)-p_{n-1}^{(h)}(c),
\end{align*}

which completes the proof in the case $i_{n-1}\neq i_n-1$.

If $i_{n-1}= i_n-1$, then $A_n = \Omega$, so the integration by parts step (applying \eqref{integral2}) is no longer needed. We can therefore carry out a slightly shorter argument with one less integral. First note that
\begin{align*}
     &p_n^{(h)}(c):=\int\limits_0^1\int\limits_0^1 \frac{(1-v_{n}+c)^{h-i_{n}-1}}{(h-i_{n}-1)!} \cdot \mathbb{P}\left(\bigcap_{k=1}^n A_k \cap \bigcap_{\ell=1}^n B_\ell \cap \{U_{i_{n}+1} \in dv_n,\, U_{i_n} \in du_n\}\right)\\
      & =\int\limits_0^1\int\limits_{u_n-\theta}^{u_n} \mathbb{P}\left(\bigcap_{k=1}^{n-1} A_k \cap \bigcap_{\ell=1}^{n-2} B_\ell \cap \{U_{i_{n-1}}\in (u_{n}, u_{n}+\theta)\}\right)\frac{(1-v_n+c)^{h-i_{n}-1}}{(h-i_{n}-1)!} d v_{n}d u_{n}.
\end{align*}
By \eqref{integral1}, this is equal to
\begin{align*}
       &\int\limits_0^1 \mathbb{P}\left(\bigcap_{k=1}^{n-1} A_k \cap \bigcap_{\ell=1}^{n-2} B_\ell\cap \{U_{i_{n-1}}\in (u_{n}, u_{n}+\theta)\}\right)\\
     &\hspace{60mm}\cdot\left[\frac{(1-u_n+\theta+c)^{h-i_{n}}}{(h-i_{n})!}-\frac{(1-u_n+c)^{h-i_{n}}}{(h-i_{n})!}\right] d u_{n}\\
     &= \int\limits_0^1\int\limits_{0}^{1} \mathbbm{1}_{\{u_n \in (u_{n-1}-\theta,u_{n-1})\} }\mathbb{P}\left(\bigcap_{k=1}^{n-2} A_k \cap \bigcap_{\ell=1}^{n-2} B_\ell\cap \{U_{i_{n-2}+1}<\ldots <U_{i_{n-1}-1}<u_{n-1}\}\right)\\
     &\hspace{45mm}\cdot\left[\frac{(1-u_n+\theta+c)^{h-i_{n}}}{(h-i_{n})!}-\frac{(1-u_n+c)^{h-i_{n}}}{(h-i_{n})!}\right] d u_{n} d u_{n-1}.
\end{align*}
Now noting that $i_n = i_{n-1}-1$, and changing the dummy variable from $u_n$ to $v_{n-1}$, the above equals
\begin{align*}     
     &\int\limits_0^1\int\limits_{0}^{1} \mathbbm{1}_{\{v_{n-1} \in (u_{n-1}-\theta,u_{n-1})\} } \mathbb{P}\left(\bigcap_{k=1}^{n-2} A_k \cap \bigcap_{\ell=1}^{n-2} B_\ell\cap \{U_{i_{n-2}+1}<\ldots <U_{i_{n-1}-1}<u_{n-1}\}\right)\\
     &\hspace{28mm}\cdot\left[\frac{(1-v_{n-1}+\theta+c)^{h-i_{n-1}-1}}{(h-i_{n-1}-1)!}-\frac{(1-v_{n-1}+c)^{h-i_{n-1}-1}}{(h-i_{n-1}-1)!}\right] dv_{n-1}du_{n-1}\\
       &= \int\limits_0^1\int\limits_0^1 \left[\frac{(1-v_{n-1}+c+\theta)^{h-i_{n-1}-1}}{(h-i_{n-1}-1)!}-\frac{(1-v_{n-1}+c)^{h-i_{n-1}-1}}{(h-i_{n-1}-1)!}\right]\\
       & \hspace{40mm} \cdot \mathbb{P}\left(\bigcap\limits_{k=1}^{n-1}A_k \cap \bigcap\limits_{\ell=1}^{n-1}B_\ell \cap \{U_{i_{n-1}+1} \in dv_{n-1}, U_{i_{n-1}} \in du_{n-1}\}\right) \\
      & =  p_{n-1}^{(h)}(c + \theta) - p_{n-1}^{(h)}(c),
\end{align*}
as required.
\end{proof}

We can now use the recursion from Lemma \ref{lemmapreinductionnegative} to prove our desired upper bound on the probability that $n$ specified uniform random variables are ``out of order'' but still within $\theta$ of the following uniform.

\begin{proposition}
\label{proposition6.3}
    Let $n \in \mathbb{N}$, $h\geq n+1$ and $c \in [0,\infty)$. Let $i_0=-1$ and $i_1<\ldots<i_n\in\{0,\ldots,h-1\}$. Then
    \begin{multline*}
    \mathbb{P}\left(\bigcap\limits_{k=1}^n\{U_{i_{k-1}+1}<\ldots <U_{i_k}\}\cap \bigcap\limits_{\ell=1}^{n}\{U_{i_\ell}\in(U_{i_\ell+1},U_{i_\ell+1}+\theta)\}\cap \{U_{i_{n}+1}<\ldots <U_h\}\right)\\
        \leq \sum\limits_{j=0}^{n}{n\choose j}(-1)^{n-j}\frac{(1+j\theta)^{h+1}}{(h+1)!}.
    \end{multline*}
\end{proposition}

\begin{proof}
Recalling the definitions of $A_k$ and $B_\ell$, and integrating out the values of $U_{i_n}$ and $U_{i_n+1}$, we have 
    \begin{align*}
        &\mathbb{P}\left(\bigcap\limits_{k=1}^{n}\{U_{i_k}\in(U_{i_k+1},U_{i_k+1}+\theta)\}\cap \bigcap\limits_{l=1}^n\{U_{i_{l-1}+1}<\ldots <U_{i_l}\}\cap \{U_{i_{n}+1}<\ldots <U_h\}\right)\\
        & = \int\limits_0^1\int\limits_0^1 \frac{(1-v_{n})^{h-i_{n}-1}}{(h-i_{n}-1)!} \mathbb{P}\left(\bigcap\limits_{k=1}^{n}A_k \cap \bigcap\limits_{\ell=1}^{n}B_\ell \cap\{U_{i_{n}+1} \in dv_n, U_{i_n} \in du_n\}\right)\\
      &= p_{n}^{(h)}(0).
    \end{align*}
By Lemma \ref{lemmapreinductionnegative}, we obtain
 \begin{align}
   p_{n}^{(h)}(0)
    &\leq p_{n-1}^{(h)}(\theta) -p_{n-1}^{(h)}(0)\nonumber\\
    &\leq p_{n-2}^{(h)}(2\theta)-2p_{n-2}^{(h)}(\theta)+p_{n-2}^{(h)}(0)\nonumber\\
    &\leq p_{n-3}^{(h)}(3\theta)-3 p_{n-3}^{(h)}(2\theta)+ 3 p_{n-3}^{(h)}(\theta)- p_{n-3}^{(h)}(0)\nonumber\\
    & \quad \vdots\nonumber\\
    &\leq \sum\limits_{j=0}^{n-1}{n-1\choose j}(-1)^{n-1-j}p_1^{(h)}(j\theta ).\label{eq:phmsum}
\end{align}
Now,
\begin{align*}
p_1^{(h)}(j\theta) &= \int\limits_0^1\int\limits_0^1 \frac{(1-v_{1}+j\theta)^{h-i_{1}-1}}{(h-i_{1}-1)!} \mathbb{P}\left(A_1\cap B_1\cap\{U_{i_{1}+1} \in dv_1, U_{i_1} \in du_1\}\right)\\
&=\int\limits_0^1\int\limits_{u_1-\theta}^{u_1} \frac{u_1^{i_1}}{i_1!} \frac{(1-v_{1}+j\theta)^{h-i_{1}-1}}{(h-i_{1}-1)!} dv_1 du_1.
\end{align*}
By \eqref{integral1}, this equals
\[\int\limits_0^1 \frac{u_1^{i_1}}{i_1!} \left[\frac{(1-v_{1}+(j+1)\theta)^{h-i_{1}}}{(h-i_{1})!}-\frac{(1-v_{1}+j\theta)^{h-i_{1}}}{(h-i_{1})!}\right] du_1\]
and then by \eqref{integral2}, we obtain that
\[p_1^{(h)}(j\theta)\le \frac{(1+(j+1)\theta)^{h+1}}{(h+1)!}-\frac{(1+j\theta)^{h+1}}{(h+1)!}.\]
Substituting this into \eqref{eq:phmsum} and by the recursive relation of the binomial coefficients, we have that
\begin{align*}
p_n^{(h)}(0) & \leq \sum\limits_{j=0}^{n-1}{n-1\choose j}(-1)^{n-1-j}\left[\frac{(1+(j+1)\theta)^{h+1}}{(h+1)!}-\frac{(1+j\theta)^{h+1}}{(h+1)!}\right]\\
    & = \sum\limits_{j=1}^{n}{n-1\choose j-1}(-1)^{n-j}\frac{(1+j\theta)^{h+1}}{(h+1)!}+\sum\limits_{j=0}^{n-1}{n-1\choose j}(-1)^{n-j}\frac{(1+j\theta)^{h+1}}{(h+1)!}\\
    &=\sum\limits_{j=1}^{n-1}{n\choose j}(-1)^{n-j}\frac{(1+j\theta)^{h+1}}{(h+1)!} + \frac{(1+n\theta)^{h+1}}{(h+1)!} + \frac{1}{(h+1)!}(-1)^{n}\\
    &=  \sum\limits_{j=0}^{n}{n\choose j}(-1)^{n-j}\frac{(1+j\theta)^{h+1}}{(h+1)!}
\end{align*}
as claimed.
\end{proof}

Finally, we sum over the possible choices of ``out of order'' uniforms to prove our main result for this section.

\begin{proof}[Proof of Proposition \ref{prop:path_upper_bd}]
In order to have $U_0<U_1+\theta < U_2 + 2\theta < \ldots < U_h + h\theta$, there must exist some $n$ between $0$ and $h$ such that $n$ of the uniforms are ``out of order'' and the remaining uniforms are ``in order'', in the sense above. By Proposition \ref{proposition6.3}, we can bound the probability of this happening by a quantity that does not depend on which indices are out of order, only on $n$. Since, for each $n$, there are at most $h\choose n$ possible choices of the $n$ ``out of order'' indices, by Proposition \ref{proposition6.3} we have
 \begin{align*}
    \mathbb{P}(U_0<U_1+\theta < \ldots < U_h + h\theta)&\leq \sum\limits_{n=0}^{h}{h\choose n}\mathbb{P}(\text{specified } n\text{ uniforms are out of order}) \\
    &\leq \sum\limits_{n=0}^{h}{h\choose n}\sum\limits_{j=0}^{n}{n\choose j} (-1)^{n-j}\frac{(1+\theta j)^{h+1}}{(h+1)!}.
 \end{align*}
Rearranging, we have 
 \begin{equation}\label{diagonalsum}
    \mathbb{P}(U_0<U_1+\theta < \ldots < U_h + h\theta)
    \leq \sum\limits_{j=0}^{h}\frac{(1+\theta j)^{h+1}}{(h+1)!}\sum\limits_{n=j}^{h}{h\choose n} {n\choose j}(-1)^{n-j}.
 \end{equation}
We note that ${h \choose n}{n\choose j} = {h\choose j}{h-j \choose n-j}$ (in both sides of the equation we are choosing disjoint sets of size $j$ and $n-j$ from $h$). Thus
 \[\sum\limits_{n=j}^{h}{h\choose n} {n\choose j}(-1)^{n-j} = {h\choose j} \sum_{n=0}^{h-j} {h-j \choose n}(-1)^{n} = {h\choose j} (1-1)^{h-j}.\]
We then see that the final sum in \eqref{diagonalsum} is zero except when $j=h$, in which case it equals $1$. The result follows immediately.
\end{proof}

\section{Explicit bounds for Rough Mount Fuji accessibility percolation on trees}\label{sec:explicit_bounds}

In this section, we apply Proposition \ref{prop:path_upper_bd}, our upper bound on the probability that one particular path of length $h+1$ is increasing, to bound the first moment of the number of increasing paths to height $h$ in a tree, and therefore show that there is no RMF accessibility percolation when $\theta$ is small. It is slightly easier to first consider Bienaym\'e-Galton-Watson trees, and we do this in Section \ref{sec:no_perco_GW}, before generalising to a wider class of trees in Section \ref{sec:no_perco_general_trees}. We then show that there is RMF accessibility percolation on Bienaym\'e-Galton-Watson trees when $\theta$ is sufficiently large via a simple multitype branching process approximation together with a symmetrization argument, in Section \ref{sec:perco_explicit_GW}.

\subsection{No accessibility percolation for small $\theta$ on Bienaym\'e-Galton-Watson trees: Proof of Proposition \ref{prop:no_perco_GW}}\label{sec:no_perco_GW}
Consider a Bienaym\'e-Galton-Watson tree whose offspring distribution $L$ satisfies $\E[L]=m\in(1,\infty)$. Let $P_h$ be the set of paths from the root to generation $h$, so that $\E[\#P_h]= m^h $. Let $N_h$ be the number of \emph{increasing} paths in $P_h$, i.e.~the number of paths $(u_0,u_1,\ldots,u_h)\in P_h$ such that $X_{u_0}<X_{u_1}<\ldots<X_{u_h}$. By Markov's inequality the probability that there exists an increasing path to generation $h$ satisfies
\[\mathbb{P}(N_h \geq 1) \leq \mathbb{E}[N_h],\]
and by the many-to-one or first moment formula,
\[\E[N_h] = m^h \mathbb{P}(\text{one fixed path of length $h$ is increasing}).\]
But for any $\theta\in [0,1]$, by Proposition \ref{prop:path_upper_bd} we have
\[\mathbb{P}(\text{one fixed path of length $h$ is increasing})=\mathbb{P}(U_0< U_1 + \theta <\ldots < U_h+ h\theta) \leq \frac{(1+ \theta h)^{h+1}}{(h+1)!}.\]
Thus
\[\mathbb{P}(N_h \geq 1)\leq m^h \frac{(1+ \theta h)^{h+1}}{(h+1)!}.\]
We now apply Stirling's approximation \cite{robbins1955remark}
\[n! \ge \frac{n^n \sqrt{2\pi n}}{e^n}\]
to see that
\[\mathbb{P}(N_h \geq 1)\leq  \frac{m^h(1+ \theta h)^{h+1}e^{h+1}}{(h+1)^{h+1}\sqrt{2\pi(h+1)}} \le \frac{1}{\sqrt{2\pi(h+1)}} \left(\frac{em}{h+1}+em\theta\right)^{h+1}.\]
We observe that this converges to zero if $\theta \le \frac{1}{em}$, and taking a logarithm gives the desired bound on the $\limsup$, completing the proof of Proposition \ref{prop:no_perco_GW}.

\subsection{No accessibility percolation for small $\theta$ on general trees: Proof of Theorem \ref{thm:no_perco_small_theta}}\label{sec:no_perco_general_trees}

Our aim in this section is to prove Theorem \ref{thm:no_perco_small_theta}, which gives an upper bound on the probability of increasing paths of length $h$ in general infinite, locally finite trees. The intuition is the same as for regular trees above, but using the definition of the branching number to show that it can be used as a replacement for the parameter $m$.

We assume for the rest of this section that $T$ is an infinite, locally finite tree with root $\rho$. We recall from Definitions \ref{definitionofcutset} and \ref{definitionofbranchingnumber} our characterisation of the branching number of $T$: first we define $\Pi\subset V$ to be a \emph{cutset} separating $\rho$ from infinity if every infinite (simple) path from $\rho$ must include at least one vertex from $\Pi$. Then the branching number of $T$ is
\[\operatorname{br} T=\sup \left\{\lambda : \inf _{\Pi} \sum_{v \in \Pi} \lambda^{-|v|}>0\right\}\]
where $|v|$ denotes the graph distance from the root $\rho$ to $v$, and the infimum is over cutsets $\Pi$ that separate $\rho$ from infinity. Clearly $\br T\ge 1$.

\begin{proof}[Proof of Theorem \ref{thm:no_perco_small_theta}]
For a cutset $\Pi$, let $N_{\Pi}$ be the number of increasing paths from $\rho$ to $\Pi$; that is,
\[\#\{v\in\Pi : \text{there exists a path $\rho=v_0,v_1,\ldots,v_{|v|} = v$ such that } X_{v_0}<X_{v_1}<\ldots<X_{v_{|v|}}\}.\]
Then for any $\theta\in[0,1]$, by Markov's inequality,
\[\mathbb{P}(N_{\Pi}\ge 1) \leq \mathbb{E}[N_{\Pi}] = \sum_{v\in \Pi} \mathbb{P}(\text{path from $\rho$ to $v$ is increasing}).\]
By Proposition \ref{prop:path_upper_bd},
\begin{equation}\label{eq:prop_cutset_ub}
\mathbb{P}(N_{\Pi}\ge 1) \leq \sum_{v\in\Pi} \frac{(1+\theta |v|)^{|v|+1}}{(|v|+1)!}.
\end{equation}

Now suppose that $\theta< \frac{1}{e\br T}$, and take $\lambda \in (\br T, \frac{1}{e\theta})$. By the definition of $\br T$, we may choose a sequence of cutsets $\Pi_n$ such that $\lim_{n\to\infty}\sum_{v\in \Pi_n} \lambda^{-|v|} = 0$. Note that if we define $I_n = \inf\{|v| : v\in \Pi_n\}$, then $\sum_{v\in\Pi_n}\lambda^{-|v|} \ge \lambda^{-I_n}$ and therefore $I_n\to\infty$ as $n\to\infty$. Thus $e/I_n + e\theta < 1/\lambda$ for all sufficiently large $n$.

It is well-known (see e.g.~\cite{robbins1955remark}) that $n! \ge n^n e^{-n}$ for all $n\in\mathbb{N}$. Thus, applying \eqref{eq:prop_cutset_ub} to our chosen cutsets $\Pi_n$ for sufficiently large $n$,
\[\mathbb{P}(N_{\Pi_n}\ge 1) \leq \sum_{v\in\Pi_n} \left(\frac{e+e\theta |v|}{|v|+1}\right)^{|v|+1} \le \sum_{v\in\Pi_n}\left(\frac{e}{|v|}+e\theta\right)^{|v|+1} < \sum_{v\in\Pi_n}\lambda^{-|v|}.\]
By our choice of cutsets, this converges to $0$ as $n\to\infty$. Since any infinite path must intersect $\Pi_n$ for every $n$, we have
\[\P(\exists \text{ infinite increasing path}) \le \P(N_{\Pi_n}\ge 1) \to 0,\]
completing the proof.
\end{proof}

\subsection{Accessibility percolation for large $\theta$ on Bienaym\'e-Galton-Watson trees: Proof of Proposition \ref{prop:perco_GW}}\label{sec:perco_explicit_GW}

Foreshadowing our approach in the next section (but following a simpler approximation to obtain a more explicit bound), we consider the uniform random variable assigned to a vertex $v$ as its type; but to avoid the complications of an uncountable type space, we discretise. Suppose that $T$ is a Bienaym\'e-Galton-Watson tree with offspring distribution $L$ satisfying $\E[L\log_+ L]<\infty$. Fix $k\in\mathbb{N}$, which we will take to be large, and assume that $\theta\ge 1/k$. If $U_v$ is the uniform random variable assigned to $v\in T$, then say that $v$ has type $Y_v = \lfloor k U_v \rfloor$. Let $T'$ be the subtree of $T$ consisting of vertices whose type $Y_v$ is at least their parent's type minus $\theta':=\lfloor k\theta\rfloor-1$.

Note that if $u$ is the parent of $v$, and $Y_v \ge Y_u - \theta'$, then $kU_v > kU_u - k\theta$ and therefore $X_v > X_u$.

We claim that $T'$ is a multitype branching process in the classical sense; see e.g.~\cite[Chapter V]{athreya_ney:branching_processes}. Indeed, the type space $I$ is $\{0,1,\ldots, k-1\}$, and the number $L^{(i)}_j$ of type $j$ children (in $T'$) of a type $i$ particle satisfies
\[\P(L^{(i)}_j = a) = \begin{cases} 0 & \text{ if } j<i-\theta'\\
									\sum_{b=a}^\infty \P(L=b){b\choose a}(1/k)^a (1-1/k)^{b-a} & \text{ if } j\ge i-\theta'.\end{cases}\]
In other words, conditional on the number of children in $T$ being $b$, the number of type $j$ children in $T'$ is Binomial$(b,1/k)$ for each $j\ge i-\theta'$, and zero otherwise. We therefore have
\[m_{ij} := \E[L^{(i)}_j] = \begin{cases} 0 & \text{ if } j<i-\theta'\\
										 \E[L]/k & \text{ if } j\ge i-\theta'.\end{cases}\]
Since $\E[L\log_+ L]<\infty$, we may apply \cite[Theorem V.6.1]{athreya_ney:branching_processes}, the multitype analogue of the Kesten-Stigum theorem. It is easily checked that the process is non-singular (not every particle has exactly one offspring), and the mean matrix $M$ consisting of the $m_{ij}$ above is positive regular, since $M^j$ has all strictly positive entries for sufficiently large $j$, and therefore has a positive maximal eigenvalue $\lambda$. Then \cite[Theorem V.6.1]{athreya_ney:branching_processes} says that $T'$ survives with positive probability if $\lambda>1$.

Consider the symmetrized matrix $\tilde M$ consisting of entries
\[\tilde m_{ij} := \E[L^{(i)}_j] = \begin{cases} 0 & \text{ if } j<i-\theta'\\
												 0 & \text{ if } j>i+\theta'\\
										 \E[L]/k & \text{ otherwise.}\end{cases}\]
It is straightforward to couple our multitype branching process above with a multitype branching process with mean matrix $\tilde M$, such that our process is always at least as large as the process corresponding to $M'$. Thus, again applying \cite{athreya_ney:branching_processes}[Theorem V.6.1], we have positive probability of survival if $M'$ has leading eigenvalue $\lambda>1$. Since $M'$ is a symmetric Toeplitz band matrix, we could calculate the eigenvalues of $M'$ exactly, but we do not lose much further accuracy by applying the min-max theorem.

\begin{proof}[Proof of Proposition \ref{prop:perco_GW}]
We apply a special case of the Courant-Fischer theorem, sometimes known as the min-max theorem; see for example \cite[Theorem 4.2.6]{horn2012matrix}. Let $\lambda$ be the largest eigenvalue of the matrix $\tilde M$ defined above. Then, by the min-max theorem, we have
\[\lambda = \max_{0\neq x \in\mathbb{R}^k} \frac{x\cdot \tilde M x}{x\cdot x}.\]
In particular taking $x$ to be the vector whose entries are all $1$, we have
\begin{align*}
\lambda \ge \frac{x\cdot \tilde M x}{k} &= \frac{\E[L]}{k^2} \left( \sum_{i=1}^{\theta'-1}(\theta'+i) + (k-2\theta'+2)(2\theta'-1)\right)\\
&= \frac{\E[L]}{k^2}\big( -(\theta')^2 + (2k+3)\theta' - (k+2) \big).
\end{align*}
Substituting $\theta' = \lfloor k\theta\rfloor-1$, we see that by taking $k$ sufficiently large we can make $\lambda>1$ if $\E[L]\theta(2-\theta)>1$. Solving for $\theta$ completes the proof.
\end{proof}

\section{Proof of Theorem \ref{theomultitype}: characterising the critical value using a multitype branching process approach}\label{sec:characterisation}

To prove Theorem \ref{theomultitype}, we again view the number of vertices in generation $h$ whose paths from the root are increasing as a multitype branching process; but now, instead of counting the number of such paths directly, and discretising and bounding from below as we did in Section \ref{sec:perco_explicit_GW}, we instead use a martingale argument to obtain an accurate asymptotic growth rate.

For vertex $v$ in our Bienaym\'e-Galton-Watson tree, we assign it a \emph{type} $(U_v,|v|)$, where $U_v$ is simply the i.i.d.~Uniform$(0,1)$ random variable associated to $v$ and $|v|$ is the generation of $v$ in the tree (or equivalently, its distance from the root).

To demonstrate with a simple example, consider the case when $T$ is a regular tree where every vertex has $m$ children for some integer $m\ge 2$. Let $\mathcal N_h$ be the set of vertices in generation $h$ whose paths from the root are increasing. Then a vertex of type $(u,h)$ that falls within $\mathcal N_h$ has a random number of children in $\mathcal N_{h+1}$ whose distribution is $\operatorname{Bin}\hspace{-0.7mm}\big(m, (1-u + \theta)\wedge 1\big)$, and each of these children has an i.i.d.~type, whose first component is uniform on $((u -\theta)\vee 0, 1)$. Thus the sequence $(\mathcal N_h,\,h\ge 0)$ forms a multitype branching process \cite{athreya_ney:branching_processes}.

\subsection{The theorem of Biggins and Kyprianou}

Multitype branching processes have been investigated by many authors. A large proportion of the literature is restricted to cases where the collection of types is finite or at least countable (hence the discretisation in Section \ref{sec:perco_explicit_GW}), whereas our model obviously has a continuum of types. To show that there are accessible paths above $\theta_c$, we will apply a very general result of Biggins and Kyprianou \cite{biggins2004measure}. We will then use a classical martingale argument to show that there are no accessible paths in the critical and subcritical cases.

In order to state the theorem of Biggins and Kyprianou that we will use, we translate the concept of a \emph{mean-harmonic} function from their paper into our setting. A function $H:[0,1]\times \mathbb{Z}_+ \to [0,\infty)$ is called mean-harmonic if for all $n$,
\[\E\left[\left.\sum_{v\in \mathcal N_1} H(U_v,n+1)\,\right|\,U_\rho = u\right] = H(u,n).\]
These mean-harmonic functions are conserved on average under reproduction and they can be used to construct natural additive martingales.

Biggins and Kyprianou also use a mean-harmonic function to define a Markov chain $\zeta$ on their type space, which represents the evolution of a single particle in the branching process. Since our application of their theorem will be quite simple, we will not actually need the definition of $\zeta$, but in order to state their theorem we include it (adapted to our context). Given a mean-harmonic function $H$, we let $\zeta = (\zeta_0,\zeta_1,\ldots)$ be a Markov chain taking values in $(0,1)$, with
\[\P(\zeta_{n+1}\in(a,b)\,|\,\zeta_n = u) = \frac{1}{H(u,n)}\E\left[ \sum_{v\in \mathcal N_1} H(U_v,n+1) \ind_{\{U_v\in(a,b)\}}\right]\]
for all $0<a<b<1$ and $n\in\{0,1,2,\ldots\}$.

Let $\mathcal{F}_n$ be the $\sigma$-algebra generated by the first $n$ generations. We can now state the theorem of Biggins and Kyprianou \cite{biggins2004measure} that we will use.

\begin{theorem}[Biggins and Kyprianou]
\label{Kyprianou theorem}
Suppose that $H:[0,1]\times \mathbb{Z}_+ \to [0,\infty)$ is a mean-harmonic function. For $x>0$, let
$$
A(x)=\sum_{i=1}^{\infty} \mathbbm{1}_{\left\{H\left(\zeta_i,i\right) x>1\right\}}.
$$
where $\zeta$ is the Markov chain defined above. Define 
$$W_n=\sum\limits_{v \in \mathcal{N}_n} H\left(U_v, n\right), \text{ and } X=\frac{W_1}{W_0} \mathbbm{1}_{\left(W_0>0\right)}+\mathbbm{1}_{\left(W_0=0\right)}.$$
Suppose that $L$ is a positive increasing function that is slowly varying at infinity. Suppose also that there is a random variable $X^*$ with
$$
\mathbb{P}(X>x \mid U_\rho=u) \leq \mathbb{P}\left(X^*>x\right) \quad \text { for all } u \in [0,1]
$$
and that $\sup _{x>0}\{A(x) / L(x)\}$ is bounded above, almost surely. If $\mathbb{E}\left[X^* L\left(X^*\right)\right]<\infty$, then $\mathbb{E}\left[\limsup_{n\to\infty}W_n | \mathcal{F}_0\right] = W_0$.
\end{theorem}

Again we note that in our application of Theorem \ref{Kyprianou theorem} we will not actually need the explicit expression of the transition measure of the Markov chain $\zeta$. Instead, we will take a trivial upper bound for $A(x)$ using the maximal possible value for $H(U_v,i)$ over all particles $v\in \mathcal N_i$. In the case where the underlying offspring distribution $L$ has a finite second moment, we could apply a standard second moment argument in place of Theorem \ref{Kyprianou theorem}. Even in the more general case $\E[L\log_+ L]$ considered in Theorem \ref{theomultitype}, we could apply the spine technique developed by Lyons, Pemantle and Peres \cite{lyons1995conceptual} and refined by Kyprianou \cite{kyprianou2004travelling} and Hardy and Harris \cite{hardy2009spine} to give a self-contained proof of our result, but this would require setting up additional notation, including a change of measure, so we take the shorter option of applying Theorem \ref{Kyprianou theorem}.

\subsection{An expression for the eigenfunctions of the increasing reproduction operator}

We need to find appropriate mean-harmonic functions so that we can apply Theorem \ref{Kyprianou theorem}. For this, we first derive a general expression that eigenfunctions $f_{m,\theta,\lambda}$ of a certain operator must satisfy, using the first-moment formula for Bienaym\'e-Galton-Watson trees.

\begin{proposition}
\label{eigenfuctiongeneralexp}
Let $m\in(0,\infty)$ and $\theta\in[0,1]$. Suppose that $\lambda\in\mathbb{R}$, and that $f_{m,\theta,\lambda}:[0,1]\to\R$ is an integrable function satisfying $f_{m,\theta,\lambda}(0)=1$ and
\begin{equation}\label{eq:efn_defn}
\mathbb{E}\left[\left.\sum \limits_{v \in \mathcal{N}_1} f_{m,\theta,\lambda}\left(U_v\right) \right| \mathcal{F}_0\right]=\lambda f_{m,\theta,\lambda}\left(U_\rho\right)
\end{equation}
almost surely. Then for any $j\in \{0,1,2,\ldots\}$ and $u \in [j\theta, (j+1)\theta)\cap(0,1]$,
\begin{equation}\label{eq:efn_recurse}
f_{m,\theta,\lambda}(u)=\sum_{i=0}^{j} \frac{(-1)^{i}m^{i}(u-i\theta)^{i}}{\lambda^{i}i!}.
\end{equation}
Moreover, any $f:[0,1]\to\mathbb{   R}$ and $\lambda\neq 0$ satisfying \eqref{eq:efn_recurse} and
\[\frac{m}{\lambda}\int_0^1 f(s) ds = f(0) = 1\]
is a normalised eigenfunction-eigenvalue pair, i.e.~satisfies \eqref{eq:efn_defn}.
\end{proposition}

Often we will fix $m$ and $\theta$ and use the notation $f_\lambda$ rather than $f_{m,\theta,\lambda}$ to keep the notation manageable.

\begin{proof}[Proof of Proposition \ref{eigenfuctiongeneralexp}]
By the first moment formula for Bienaym\'e-Galton-Watson trees (sometimes known as the many-to-one lemma), for any integrable $f:[0,1]\to\mathbb{R}$ we have
\begin{equation}\label{eq:manytoone}
\mathbb{E}\left[\left.\sum \limits_{v \in \mathcal{N}_{1}} f\left(U_v\right) \right| \mathcal{F}_0\right] = m\int_{(U_{\rho}-\theta)\vee 0}^1 f(s) ds.
\end{equation}
Thus the condition \eqref{eq:efn_defn} holds for integrable $f$ if and only if
\begin{equation}\label{eq:f_integral_eqn}
f(u)=\frac{m}{\lambda} \int_{(u-\theta)\vee 0}^1 f(s) d s \quad \forall u \in(0,1) .
\end{equation}

Suppose now that $f_\lambda$ is an integrable function with $f_\lambda(0)=1$ and satisfying \eqref{eq:efn_defn}. Then by \eqref{eq:f_integral_eqn}, for $u < \theta$ we have $f_{\lambda}(u)=f_\lambda(0) = 1$. For $u \in [\theta, 2 \theta)$, by differentiating \eqref{eq:f_integral_eqn}, $f_{\lambda}^{\prime}(u)=-\frac{m}{\lambda}$ and therefore noting that \eqref{eq:f_integral_eqn} implies that $f_\lambda$ must be continuous, we must have $f_{\lambda}(u)= 1- \frac{m}{\lambda}(u-\theta)$.

Now, as an induction hypothesis on $j$, we suppose that for $u \in [j \theta, (j+1)\theta)\cap (0,1]$ we have 
$$f_{\lambda}(u)=\sum_{k=0}^{j} \frac{(-1)^{k}m^{k}(u-k\theta)^{k}}{\lambda^kk!}.$$
Then for $u \in [(j+1) \theta, (j+2)\theta) \cap [0,1]$ we have 
\begin{align*}
   f^{\prime}_{\lambda}(u)= \frac{-f_{\lambda}(u-\theta)}{\int_{0}^1f_{\lambda}(v)dv} = \frac{\sum_{k=0}^{j} \frac{(-1)^{k+1}m^{k}(u-(k+1)\theta)^{k}}{\lambda^{k}k!}}{\lambda/m} = \sum_{k=0}^{j} \frac{(-1)^{k+1}m^{k+1}(u-(k+1)\theta)^{k}}{\lambda^{k+1}k!}.
\end{align*}
Integrating with respect to $u$,
$$f_{\lambda}(u)=\sum_{k=0}^{j} \frac{(-1)^{k+1}m^{k+1}(u-(k+1)\theta)^{k+1}}{\lambda^{k+1}(k+1)!} + C.$$

Setting $i=k+1$ we have 
$$f_{\lambda}(u)=\sum_{i=1}^{j+1} \frac{(-1)^{i}m^{i}(u-i\theta)^{i}}{\lambda^{i}i!} + C,$$
and for continuity we need 
$$f_{\lambda}((j+1)\theta)=\sum_{k=0}^{j} \frac{(-1)^{k}m^{k}((j+1 -k)\theta)^{k}}{\lambda^{k}k!}.$$
Thus, 
$$C= 1 + \sum_{k=1}^{j} \frac{(-1)^{k}m^{k}((j+1 -k)\theta)^{k}}{\lambda^{k}k!} - \sum_{i=1}^{j} \frac{(-1)^{i}m^{i}((j+1-i)\theta)^{i}}{\lambda^{i}i!}=1.$$

By induction, for $u \in [(j+1) \theta, (j+2)\theta) \cap [0,1] $ we have
$$f_{\lambda}(u)=\sum_{i=0}^{j+1} \frac{(-1)^{i}m^{i}(u-i\theta)^{i}}{\lambda^{i}i!},$$
i.e.~$f_\lambda$ satisfies \eqref{eq:efn_recurse}. 

It now remains to check the converse, i.e.~that if $f$ and $\lambda\neq 0$ satisfy \eqref{eq:efn_recurse} and $\frac{m}{\lambda}\int_0^1 f(s) ds = f(0)=1$ then they also satisfy \eqref{eq:efn_defn}. Differentiating \eqref{eq:efn_recurse} easily yields that for any $u\ge \theta$ we have
\[f'(u) = -\frac{m}{\lambda} f(u-\theta),\]
and $f'(u)=0$ for $u\in(0,\theta)$. Integrating, we obtain that for any $u\in(0,1)$,
\[f(u) = \frac{m}{\lambda}\int_{(u-\theta)\vee 0}^1 f(s) ds + C\]
for some constant $C$; thus \eqref{eq:f_integral_eqn} is satisfied if and only if $C=0$. This condition is satisfied precisely when $f(0)=\frac{m}{\lambda}\int_0^1 f(s) ds$, and we have our desired normalisation when this quantity equals $1$.
\end{proof}

We now make the simple observation that the eigenfunctions scale linearly with $m$, the expectation of the offspring distribution.

\begin{corollary}\label{cor:scaling_mlambda} 
For each $m\in(0,\infty)$ and $\theta\in[0,1]$, suppose that $\lambda_{m,\theta}$ is an eigenvalue corresponding to an eigenfunction $f_{m,\theta,\lambda}$ from Proposition \ref{eigenfuctiongeneralexp}; that is, suppose that $\lambda_{m,\theta}$ and $f_{m,\theta,\lambda}$ satisfy \eqref{eq:efn_defn}. Then for each $m\in(0,\infty)$ and $\theta\in[0,1]$, $\lambda_{m,\theta}=m\lambda_{1,\theta}$.
\end{corollary}

\begin{proof}
 Let the tree $T$ have offspring distribution $L$ with $\mathbb{E}[L]=m$, and suppose that $f_{m,\theta,\lambda}$ and $\lambda_{m,\theta}$ are the associated eigenfunction and eigenvalue. Similarly let $T^{\prime}$ be another tree with offspring distribution $L^{\prime}$ such that $\mathbb{E}[L^{\prime}]=m^{\prime}$. Then by \eqref{eq:manytoone} we have
 $$\mathbb{E}_{L^{\prime}}\left[\left.\sum_{v\in\mathcal{N}_{1}}f_{m,\theta,\lambda_{m,\theta}}(U_{v})\right|\mathcal{F}_{0}\right] = \frac{m^{\prime}}{m}\mathbb{E}_{L}\left[\left.\sum_{v\in\mathcal{N}_{1}}f_{m,\theta,\lambda_{m,\theta}}(U_{v})\right|\mathcal{F}_{0}\right]=\frac{m^{\prime}}{m}\lambda_{m,\theta}f_{m,\theta,\lambda_{m,\theta}}(U_{\rho}).$$
 This implies that $f_{m,\theta,\lambda_{m,\theta}}$ is an eigenfunction corresponding to $m'$ and $\theta$, with eigenvalue $\lambda_{m^{\prime},\theta}=\frac{m^{\prime}}{m}\lambda_{m,\theta}$. Taking $m^{\prime}=1$ completes the proof.
\end{proof}

\subsection{Warm-up: proof of survival in Theorem \ref{theomultitype} for $\theta\in(1/2,1)$}

As a warm-up for the rest of the proof of Theorem \ref{theomultitype}, we will now prove the survival part of Theorem \ref{theomultitype} in the case $\theta \in(1/2,1)$. In this case, we can carry out the calculations explicitly, before we apply a more abstract argument to attack the general case $\theta\in(0,1)$. 

Fix $\theta\in(1/2,1)$ and $m\in(1,\infty)$. By Proposition \ref{eigenfuctiongeneralexp} we have candidate eigenfunctions $f_\lambda$ satisfying
\[f_\lambda(u)=\begin{cases} 1 & \text{ for } u\in[0,\theta],\\
                             1-\frac{m}{\lambda}(u-\theta) & \text{ for } u \in(\theta,1].\end{cases}\]
Then 
\begin{align*}
    \frac{m}{\lambda} \int_0^1 f_\lambda(s) d s &=\frac{m}{\lambda} \int_0^\theta 1 d s+\frac{m}{\lambda} \int_\theta^1\left(1-\frac{m}{\lambda}(u-\theta)\right) d u \\
    &=\frac{m}{\lambda}-\frac{m^2}{2 \lambda^2}(1-\theta)^2.
\end{align*}
Now checking the condition that the integral above must equal $1$, we see that 
$$
\lambda^2-m \lambda+\frac{m^2}{2}(1-\theta)^2=0
$$
from which we obtain
$$
\lambda_\pm =\frac{m}{2} \pm \frac{m}{2} \sqrt{1-2(1-\theta)^2}.
$$
We concentrate on the largest eigenvalue $
\lambda_+ =\frac{m}{2} + \frac{m}{2} \sqrt{1-2(1-\theta)^2}
$. 
This gives, when $\theta\in(1/2,1)$,
$$
f_{\lambda_+} = f_{m,\theta,{\lambda_+}}(s)=\begin{cases}
1 & \text { if } s \le \theta \\
1-\frac{m}{\lambda_+}(s-\theta) &\text { if } s>\theta,
\end{cases}
$$
which Proposition \ref{eigenfuctiongeneralexp} guarantees is indeed a normalised eigenfunction, i.e.~satisfies \eqref{eq:efn_defn}. It is straightforward to check also that for $\theta\in(1/2,1)$, $f_{\lambda_+}$ is decreasing and $f_{\lambda_+}(1)>0$, so that in fact $f_{\lambda_+}(u)>0$ for all $u\in[0,1]$.

To be able to apply Theorem \ref{Kyprianou theorem} and prove that the process survives we define 
\[H(s,n)=\frac{1}{\lambda_+^n} f_{m,\theta,\lambda_+}(s),\]
and note that $H$ is mean-harmonic by \eqref{eq:efn_defn}. We set, as in Theorem \ref{Kyprianou theorem},
$W_n=\sum_{v \in \mathcal{N}_n} H\left(U_v,n\right)$, and note that
\begin{align*}
   A(x) := &\sum_{i=1}^{\infty} \mathbbm{1}_{\left\{H\left(\zeta_i\right) x>1\right\}}= \sum\limits_{i=1}^\infty \mathbbm{1}_{\{x f_{\lambda_+}(\zeta_i) > \lambda_+^i\}} \leq \sum\limits_{i=1}^\infty \mathbbm{1}_{\{x > \lambda_+^i\}}
\end{align*}
since $f_{\lambda_+}(u)\le 1$ for any $u\in(0,1)$ (note that, as we stated earlier, we did not use the definition of the Markov chain $\zeta$ from Theorem \ref{Kyprianou theorem} for this bound).

We choose $L(x)= 1+\log_+ x$. Let $\mathcal{N}^+_1$ be the set of children of the root without deletions, i.e.~when we keep all vertices regardless of their label. Let $X^* = |\mathcal N^+_1|/ \lambda_+ f_{\lambda_+}(1)$. One of the assumptions of Theorem \ref{theomultitype} is that $\E[|\mathcal N^+_1|\log_+ |\mathcal N^+_1|]<\infty$, and thus $\mathbb{E}[X^* L(X^*)]< \infty$.

If we define $$X=\frac{\sum_{v \in\mathcal{N}_1} H\left(U_v,1\right)}{H\left( U_\rho, 0\right)}=\frac{W_1}{W_0},$$
then
$$
X \le \frac{\frac{1}{\lambda_+} |\mathcal N^+_1|}{H(s,0)} \le \frac{|\mathcal N^+_1|}{\lambda_+ f_{\lambda_+}(1)} =X^*.
$$

Finally, if $\lambda_+>1$ then $A(x)\le \frac{\log x}{\log \lambda_+}$ and therefore $\sup_{x>0}\{A(x)/L(x)\}<\infty$, and therefore by Theorem \ref{Kyprianou theorem}, in this case we have $\mathbb{E}[\limsup_{n\to\infty} W_n | \mathcal{F}_0] =W_0 = f_{\lambda_+}(U_0) > 0$, and thus the process survives with strictly positive probability. It remains only to check that $\lambda_+>1$ corresponds to $m>m_c(\theta)$ where $m_c(\theta)$ is the smallest root of $Q_\theta(m)=0$. But we have already calculated directly that $\lambda_+ = m(1+\sqrt{1-2(1-\theta)^2})/2$, and therefore $\lambda_+>1$ precisely when
\[m> \frac{2}{1+\sqrt{1-2(1-\theta)^2}} = \frac{1-\sqrt{1-2(1-\theta)^2}}{(1-\theta)^2}\]
which is the smaller root of $Q_\theta(m) = 1-m+m^2(1-\theta)^2/2$. This completes the proof of survival in Theorem \ref{theomultitype} in the case $\theta\in(1/2,1)$.

\subsection{Survival in Theorem \ref{theomultitype} in the general case $\theta\in(0,1)$}\label{sec:survival_gen}

In the general case, for any $\theta \in (0,1)$, in order to build a mean-harmonic function we need to guarantee that the leading eigenfunction $f_{m,\theta,\lambda}$ is non-negative. We did this by explicit calculation when $\theta\in(1/2,1)$, but calculating $f_{m,\theta,\lambda}$ is increasingly difficult as $\theta$ decreases. Instead we use the Krein-Rutman theorem, a generalisation of the Perron-Frobenius theorem for compact operators. First we need to check that our operator is indeed compact. Let $C_{[0,1]}$ be the set of continuous functions from $[0,1]$ to $\mathbb{R}$, equipped with the uniform norm.

\begin{lemma}
\label{compactoperator}
    Let $T=T_{m,\theta}:C_{[0,1]}\to C_{[0,1]}$ be the linear operator such that, for any $g\in C_{[0,1]}$ and $s \in [0,1]$,
    $$T g(s)= \mathbb{E}\left[\left.\sum\limits_{v \in \mathcal{N}_1} g (U_v) \right| U_{\rho}=s\right].$$
    Then $T$ is compact.
\end{lemma}

\begin{proof}
     We must prove that if $B\subset C_{[0,1]}$ is bounded, then $T(B)$ is totally bounded. By the Arzel\`a-Ascoli theorem it is enough to show that $T(B)$ is bounded and equicontinuous.
     
     Suppose that $B \subset C_{[0,1]}$ is bounded by $M$, that is, $\Vert g \Vert \leq M$ for all $g \in B$. Then for $g \in B$,
     \[\| Tg \| = \sup_{s \in [0,1]} \left\vert\, \mathbb{E}\left[\left.\sum_{v \in \mathcal{N}_1} g (U_v) \right| U_{\rho}=s\right] \right \vert \leq M \sup\limits_{s \in [0,1]} \mathbb{E}\left[ {N}(1)  \mid U_{\rho}=s\right] = M \E[ N(1) ],\]
     so $T(B)$ is bounded.

Now, to prove $Tg$ is equicontinuous, take $s<t$. Recall that $\mathcal N^+_1$ is the set of \emph{all} children of the root, i.e.~when we keep all vertices regardless of their label. Then
\begin{align*}
     \vert T g(s)- T g(t) \vert &= \left\vert \mathbb{E}\left[\left.\sum\limits_{v \in \mathcal{N}_1} g (U_v) \right| U_{\rho}=s\right]  -\mathbb{E}\left[\left.\sum\limits_{v \in \mathcal{N}_1} g (U_v) \right| U_{\rho}=t\right] \right \vert \\
     &= \left\vert \mathbb{E}\left[\sum\limits_{v \in \mathcal N^+_1} \mathbbm{1}_{\{U_v>s-\theta\}} g (U_v)\right]  -\mathbb{E}\left[\sum\limits_{v \in \mathcal N^+_1}\mathbbm{1}_{\{U_v>t-\theta\}} g (U_v) \right] \right \vert \\
     &= \left\vert \mathbb{E}\left[\sum\limits_{v \in \mathcal N^+_1} \mathbbm{1}_{\{U_v \in (s-\theta, t -\theta]\}} g (U_v)\right] \right \vert \\
   &\leq  \mathbb{E}\left[\sum\limits_{v \in \mathcal N^+_1} \mathbbm{1}_{\{U_v \in (s-\theta, t -\theta]\}} \vert g (U_v) \vert\right] \\
   & \leq \Vert g \Vert m(t-s).
\end{align*}
We deduce from the Arzel\`a-Ascoli theorem that $T(B)$ is totally bounded, and therefore that $T$ is compact.
\end{proof}

We will now use Lemma \ref{compactoperator} to prove that $f_{m,\theta,\lambda}$ is non-negative when $\lambda$ is the lead eigenvalue of $T_{m,\theta}$.

\begin{proposition}
\label{non-negative eigenfunction}
    For any $m>1$ and $\theta\in(0,1)$, the operator $T_{m,\theta}$ defined in Lemma \ref{compactoperator} has a strictly positive largest real eigenvalue $\lambda$, and the eigenfunction $f_{m,\theta,\lambda}$ corresponding to $\lambda$ is non-negative.
\end{proposition}

\begin{proof}
Proposition \ref{eigenfuctiongeneralexp} tells us that the functions $f_{m,\theta,\lambda}$ defined by \eqref{eq:efn_recurse} are eigenfunctions of $T$ provided that $\frac{m}{\lambda}\int_0^1 f(s) ds = 1$. Applying this constraint to \eqref{eq:efn_recurse} and rearranging, we obtain the following polynomial equation in $\lambda$:
\[\sum_{i=0}^{\lfloor 1/\theta\rfloor} \frac{(-1)^i}{(i+1)!} m^{i+1}  (1-i\theta)^{i+1} \lambda^{\lfloor 1/\theta\rfloor-i} = \lambda^{\lfloor 1/\theta\rfloor+1}.\]
It is easy to see that this equation has non-zero roots, since it has order at least $2$, non-zero constant coefficient for all values of $\theta$ such that $1/\theta\not\in\mathbb{N}$, and non-zero coefficient of $\lambda$ for all values of $\theta\in(0,1]$. Thus $T$ has positive spectral radius.

    Recall that $C_{[0,1]}$ is the set of continuous functions on $[0,1]$ equipped with the uniform norm. Define $K \subset C_{[0,1]}$ to be the set of non-negative continuous functions,
    $$K=\{f\in C_{[0,1]} : f(u)\ge 0\,\,\forall u\in[0,1]\}.$$
    Then $K$ is a convex cone (a subset of a vector space closed under linear combinations with positive coefficients), $K \cap -K= \{0\}$, and the set $\{u-v: u,v \in K\}$ is equal to $C_{[0,1]}$. By Lemma \ref{compactoperator}, the operator $T$ is compact. Therefore, by the Krein–Rutman Theorem, the spectral radius of $T$ is a strictly positive real eigenvalue corresponding to a non-negative eigenfunction.
\end{proof}
 
\begin{lemma}
\label{lemmaf1positive}
Suppose that $\lambda$ is the largest real eigenvalue of $T_{m,\theta}$, whose existence is guaranteed by Proposition \ref{non-negative eigenfunction}. Then the corresponding eigenfunction $f_{m,\theta,\lambda}$ is decreasing, i.e.~$f_{m,\theta,\lambda}(u)\ge f_{m,\theta,\lambda}(s)$ whenever $0\le u\le s\le 1$; and $f_{m,\theta,\lambda}(1)>0$.
\end{lemma}

\begin{proof}
We recall \eqref{eq:f_integral_eqn}, which said that any eigenfunction must satisfy
\[f(u) = \frac{m}{\lambda} \int_{(u-\theta)\vee 0}^1 f(t) dt.\]
Since, by Proposition \ref{non-negative eigenfunction}, $f(t)\ge 0$ for all $t$, it follows that $f$ is decreasing in $u$; it also follows that $f$ is continuous. Now suppose, for a contradiction, that $f(t)=0$ for some $t\in(0,1]$. Let
\[t_0 = \inf\{t\in(0,1] : f(t) = 0\}.\]
Since $f(t)=1$ for all $t\le \theta$, and $f$ is continuous, we must have $t_0>\theta$. Then by \eqref{eq:f_integral_eqn},
\[f(t_0) = \frac{m}{\lambda}\int_{(t_0-\theta)\vee 0}^1 f(s) ds \ge \frac{m}{\lambda}\int_{(t_0-\theta)\vee 0}^{t_0} f(s) ds > 0.\]
If $t_0=1$ then we are done; on the other hand, if $t_0<1$, then by continuity (which also follows from \eqref{eq:f_integral_eqn}), there must exist $\eps>0$ such that $f(t_0+\eps)>0$, which contradicts the definition of $t_0$ since $f$ is decreasing. This completes the proof.
\end{proof}

\begin{proposition}\label{prop:m_implies_lambda}
Again let $\lambda_{m,\theta}$ be the largest real eigenvalue of the operator $T_{m,\theta}$ defined in Lemma \ref{compactoperator}. Let $m_c(\theta)$ be the smallest root of $Q_\theta(m)=0$, where $Q_\theta$ was defined in Theorem \ref{theomultitype}. Then $\lambda_{1,\theta} = 1/m_{c}(\theta)$. Consequently, by the scaling relation in Corollary \ref{cor:scaling_mlambda}, $\lambda_{m,\theta} = m/m_{c}(\theta)$, meaning $\lambda_{m,\theta}>1$ if $m>m_{c}(\theta)$, $\lambda_{m,\theta}<1$ if $m<m_{c}(\theta)$, and $\lambda_{m_{c}(\theta),\theta}=1$.
\end{proposition}

\begin{proof}
    Let $P_{\theta}(m,\lambda)=\frac{m}{\lambda}\int_{0}^{1}f_{m,\theta,\lambda}(s)ds-1$, where $f_{m,\theta,\lambda}$ is the function given by \eqref{eq:efn_recurse}. Let $\overline{m}_{c}(\theta)$ be the smallest $m$ such that $P_{\theta}(m,1)=0$. First, we show that $\overline{m}_{c}(\theta)=m_{c}(\theta)$. We want to show that the smallest $m$ such that $P_{\theta}(m,1)=0$ is also the smallest $m$ such that $Q_{\theta}(m)=0$. By algebraic manipulation of \eqref{eq:efn_recurse},
\begin{align*}
   P_\theta(m,1) = m \int^1_0 f_{m,\theta,1}(s)ds - 1 & = m \sum_{j=0}^{\lfloor 1/\theta\rfloor} \int_{j\theta}^{(j+1)\theta\wedge 1} f_{m,\theta,1}(s)ds - 1\\
   & = m \sum_{j=0}^{\lfloor 1/\theta\rfloor} \int_{j\theta}^{(j+1)\theta\wedge 1} \sum_{i=0}^{j} \frac{(-1)^{i}m^{i}(s-i\theta)^{i}}{i!} ds - 1\\    & = m\sum_{i=0}^{\lfloor 1/\theta\rfloor} \sum_{j=i}^{\lfloor 1/\theta\rfloor} \int_{j\theta}^{(j+1)\theta\wedge 1} \frac{(-1)^{i}m^{i}(s-i\theta)^{i}}{i!} ds - 1\\
   & =  m\sum_{i=0}^{\lfloor 1/\theta\rfloor} \frac{(-1)^{i}m^{i}}{i!} \int_{i\theta}^{1} (s-i\theta)^{i} ds - 1\\
   & = m\sum_{i=0}^{\lfloor 1/\theta\rfloor} \frac{(-1)^{i}m^{i}}{(i+1)!}(1-i\theta)^{i+1} - 1\\
   & = \sum_{k=0}^{\lfloor 1/\theta\rfloor + 1} \frac{(-1)^{k-1} m^k}{k!} (1-(k-1)\theta)^k = -Q_\theta(m).
\end{align*}
Thus $\overline{m}_{c}(\theta)=m_{c}(\theta)$. This implies that when $m=m_{c}(\theta)$, $\lambda=1$ is a root of $P_{\theta}(m_{c}(\theta),\lambda)=0$. Therefore, the lead eigenvalue for $m_{c}(\theta)$ is $\lambda_{m_{c}(\theta),\theta}=1$. By the scaling relation from Corollary \ref{cor:scaling_mlambda}, we then know that $1=\lambda_{m_{c}(\theta),\theta}=m_{c}(\theta)\lambda_{1,\theta}$. It then immediately follows from the scaling relation that $\lambda_{m,\theta}=m\lambda_{1,\theta}=m/m_{c}(\theta)$. Thus, $\lambda_{m,\theta}>1$ when $m>m_{c}(\theta)$, and $\lambda_{m,\theta}<1$ when $m<m_{c}(\theta)$.
\end{proof}

\begin{proof}[Proof of survival in Theorem \ref{theomultitype}: there is RMF accessibility percolation 
when $m>m_c(\theta)$]

The proof is now very similar to the case $\theta>1/2$. As in the $\theta>1/2$ case, we let $H(s,n)=\frac{1}{\lambda^n} f_{m,\theta,\lambda}(s)$ and $W_n=\sum_{v \in \mathcal{N}_n} H\left(U_v,n\right)$, and note that
\[
A(x)= \sum_{i=1}^{\infty} \mathbbm{1}_{\left\{H\left(\zeta_i\right) x>1\right\}}= \sum\limits_{i=1}^\infty \mathbbm{1}_{\{x f_{m,\theta, \lambda}(\zeta_i) >\lambda^i\}} \leq \sum\limits_{i=1}^\infty \mathbbm{1}_{\{x > \lambda^i\}}. 
\]
Provided that $\lambda>1$, we have $A(x)\le \frac{\log x}{\log\lambda}$. We choose $L(x)= 1+\log_+ x$. Then clearly $A(x)/L(x)$ is bounded above provided that $\lambda>1$.

We recall that we write $\mathcal N^+_1$ for the set of children of the root without deletions. Let $X^* = |\mathcal N^+_1|/ \lambda f_{m,\theta,\lambda}(1)$. One of the assumptions of the theorem was that $\E[|\mathcal N^+_1|\log_+ |\mathcal N^+_1|]<\infty$, and therefore since $f_{m,\theta,\lambda}(1)>0$ by Lemma \ref{lemmaf1positive} we have $\mathbb{E}[X^*L(X^*)]< \infty$. If we define
$$X=\frac{\sum_{v \in\mathcal{N}(1)} H\left(U_v,1\right)}{H\left( U_\rho, 0\right)} =\frac{W_1}{W_0},$$
by Lemma \ref{lemmaf1positive} $f_{m,\theta,\lambda}$ is decreasing with $f_{m,\theta,\lambda}(1)>0$, and thus
$$
X \leqslant \frac{\frac{1}{\lambda} |\mathcal N^+_1|}{H(U_\rho,0)} \leqslant \frac{|\mathcal N^+_1|}{\lambda f_{m,\theta,\lambda}(1)} =X^*.
$$

By Proposition \ref{prop:m_implies_lambda} we know that if $m>m_c(\theta)$, then $\lambda>1$, and finally by Theorem \ref{Kyprianou theorem} we have that if $\lambda>1$, then $\mathbb{E}[\limsup\limits_{n\to\infty} W_n | \mathcal F_0] =W_0$, and therefore the process survives with strictly positive probability.
\end{proof}

\subsection{Extinction in Theorem \ref{theomultitype} }

To prove that the process does not survive, we can follow a standard approach. In the subcritical case, we can read off an exponential decay in the probability of survival from analysing the martingale $W_n$; and in the critical case, we observe that $W_n$ must converge, and show that the only possible finite limit is $0$.

\begin{proof}[Proof of extinction in Theorem \ref{theomultitype}: no RMF accessibility percolation for $m\le m_c(\theta)$]
We know by Proposition \ref{prop:m_implies_lambda} that $m\le m_c(\theta)$ corresponds to $\lambda_{m,\theta}\le 1$.

The subcritical case $\lambda_{m,\theta}<1$ (corresponding to $m<m_c(\theta)$) is almost immediate given the groundwork of Section \ref{sec:survival_gen}. We again let 
\[W_n=\sum\limits_{v \in \mathcal{N}_n} H\left(U_v, n\right)= \sum\limits_{v \in \mathcal{N}_n} \lambda^{-n} f\left(U_v\right).\]
By definition of $f$ we have that $(W_n)_{n\ge 0}$ is a martingale, and by Lemma \ref{lemmaf1positive}, $W_n \ge \lambda^{-n} N_n f(1)$. Therefore
\[\P(N_n\ge 1) \le \P(W_n \ge f(1)\lambda^{-n}) \le \frac{\E[W_n]}{f(1)\lambda^{-n}} = \frac{\lambda^n \E[W_0]}{f(1)}.\]
If $\lambda<1$, then $\P(N_n\ge 1)\to 0$, and by Proposition \ref{prop:m_implies_lambda}, this holds if $m<m_c(\theta)$. This completes the proof of Theorem \ref{theomultitype} in the case $\lambda<1$.

We now consider the critical case when the lead eigenvalue is $\lambda=1$. Then, as defined previously, $W_n = \sum_{v \in \mathcal{N}_n} f(U_v)$ is a non-negative martingale. By the martingale convergence theorem, $W_n$ converges almost surely to a random variable $W_\infty$, which is finite almost surely.

We recall from Lemma \ref{lemmaf1positive} that the eigenfunction $f$ is decreasing and satisfies $f(1)>0$. Therefore, for any $n \in \mathbb{N}$, we have the lower bound $W_n \ge N_n f(1)$. This implies that the number of accessible paths cannot grow to infinity: almost surely,
\[
\limsup_{n \rightarrow \infty} N_n \le \limsup_{n \rightarrow \infty} \frac{W_n}{f(1)} = \frac{W_{\infty}}{f(1)} < \infty.
\]

Take $\nu\in(0,1-\theta)$ and $k\in\mathbb{N}$ such that $p_k = \P(|\mathcal N^+_1|=k)>0$, i.e.~such that the offspring distribution of the underlying Galton-Watson tree gives strictly positive mass to $k$ children. We claim that for any $u\in [0,1]$,
\[\P( N_2 = 0 \, |\, U_\rho = u ) \ge p_k\cdot (1-\theta-\nu)^k\cdot p_k^k\cdot \nu^{k^2}.\]
Indeed, whatever the value of $U_\rho$, we have the possibility that the root has exactly $k$ children, all of whose uniforms are larger than $\theta+\nu$; and each of these children has exactly $k$ children, all of whose uniforms are smaller than $\nu$. This event has the probability specified above, and results in all particles being killed by generation 2.

By the Markov property we deduce that, every time $t$ that there are $j\ge 1$ particles in the system, there is probability at least $(p_k\cdot (1-\theta-\nu)^k\cdot p_k^k\cdot \nu^{k^2})^j$ that the system is extinct two generations later. Thus we cannot have
\[\limsup_{n \rightarrow \infty} N_n = j\]
for any $j\ge 1$, and therefore
\[\lim_{n \rightarrow \infty} N_n = 0,\]
completing the proof of Theorem \ref{theomultitype}.
\end{proof}

\subsection{Application to asymptotics of the probability that $h+1$ RMF labels are increasing}

We now apply our results on the calculation and estimation of the critical value for RMF accessibility percolation to prove Corollary \ref{cor:spectral_prob_est}, i.e.~to bound
\[\P(U_0<U_1+\theta<\ldots<U_h+h\theta)\]
from above and below.

\begin{proof}[Proof of Corollary \ref{cor:spectral_prob_est}]
Take $\theta\in(0,1)$. We again consider the critical case when the lead eigenvalue $\lambda=1$; we recall again that by Proposition \ref{prop:m_implies_lambda}, this corresponds to $m=m_c(\theta)$, and thus our martingale $(W_n)_{n\ge 0}$ satisfies
\[W_n = \sum_{v\in\mathcal{N}_n} f(U_v).\]
Moreover, from Lemma \ref{lemmaf1positive}, the eigenfunction $f$ is non-increasing, with $f(0)=1$ and $f(1)>0$. Thus (recalling that $N_h = |\mathcal N_h|$)
\[N_h f(1) \le W_h \le N_h.\]
Since $\E[W_h] = \E[W_0] = \E[f(U_\rho)]$, we deduce that
\[\E[f(U_\rho)] \le \E[ N_h ] \le \frac{\E[f(U_\rho)]}{f(1)}.\]

However, by the many-to-one or first moment formula,
\[\E[ N_h ] = m_c(\theta)^h \P(U_0<U_1+\theta<\ldots<U_h+h\theta),\]
and therefore
\[\frac{\E[f(U_\rho)]}{m_c(\theta)^h} \le \P(U_0<U_1+\theta<\ldots<U_h+h\theta) \le \frac{\E[f(U_\rho)]}{f(1) m_c(\theta)^h},\]
proving the first part of the result. The second part then follows from the bounds in Theorem \ref{thm:no_perco_small_theta} and Proposition \ref{prop:perco_GW}.
\end{proof}

\section{Rough Mount Fuji accessibility percolation on the lattice}

In this section, we consider Rough Mount Fuji accessibility percolation on the integer lattice $\mathbb{Z}^n$, and aim to prove the results of Section \ref{sec:zd_intro}.

We recall that in this setting, it is important to specify (a) which paths are permitted, and (b) which distance $d$ we are using in the definition of the RMF model (Definition \ref{RMF}).

We consider two cases for the permitted paths: the non-backtracking case, where the path is required to move further away from the origin with each step; or the case of all simple paths, i.e.~where paths may take some steps back towards the origin. In the latter case the system is no longer monotone in $\theta$ (see Figure \ref{fig:counterexample}) so we introduced two potentially different critical points
\[\theta_c = \inf\{\theta\in[0,1] : \P_\theta(\text{RMF accessibility percolation})>0\},\]
and
\[\theta'_c = \sup\{\theta\in[0,1] : \P_\theta(\text{RMF accessibility percolation})=0\}.\]

For the distance $d$, we will consider the $\ell^q$ distance for any $q\in[1,\infty]$; the case $q=1$ is of course the graph distance.

Our main goal is to find a non-trivial lower bound for $\theta_c$ and a non-trivial upper bound for $\theta'_c$, in all the cases mentioned above.

\subsection{Proof of Proposition \ref{prop:lattice_no_perc}: there is no Rough Mount Fuji accessibility percolation on the integer lattice for small $\theta$}

We recall Proposition \ref{prop:lattice_no_perc}, which gives lower bounds on the value of $\theta_c$ by showing that there are no infinite increasing paths when $\theta$ is sufficiently small. The proof is straightforward: we apply Proposition \ref{prop:path_upper_bd} with a union bound, similarly to the proof of Proposition \ref{prop:no_perco_GW} in Section \ref{sec:no_perco_GW}.

\begin{proof}[Proof of Proposition \ref{prop:lattice_no_perc}]

Assume that we have a graph of maximal degree $\Delta<\infty$. Since there are at most $\Delta(\Delta-1)^{h-1}$ vertices at distance $h$ from $0$, by Markov's inequality,
\begin{align*}
    \mathbb{P}( \exists \text{ an increasing path of length $h$}) & \leq \mathbb{E}[\# \text{ increasing paths of length $h$}]\\
    &\le \Delta(\Delta-1)^{h-1} \mathbb{P}(U_0< U_1 + \theta <\ldots < U_h+ h\theta).
    \end{align*}
By Proposition \ref{prop:path_upper_bd} we have
\[\mathbb{P}(U_0< U_1 + \theta <\ldots < U_h+ h\theta) \leq \frac{(1+ \theta h)^{h+1}}{(h+1)!}.\]
Thus, applying  Stirling's approximation,
\begin{align*}
   \mathbb{P}(\exists \text{ an increasing path of length $h$})& \leq \Delta(\Delta-1)^{h-1} \frac{(1+ \theta h)^{h+1}}{(h+1)!}\\
&\leq  \frac{\Delta(\Delta-1)^{h-1}(1+ \theta h)^{h+1}e^{h+1}}{(h+1)^{h+1}\sqrt{2\pi(h+1)}}\\ 
&\le \frac{1}{\sqrt{2\pi(h+1)}} \left(\frac{e(\Delta -1)}{h+1}+e(\Delta -1)\theta\right)^{h+1}. 
\end{align*}
We observe that this converges to zero if $\theta \le \frac{1}{(\Delta -1)e}$; indeed, for the boundary case $\theta = \frac{1}{(\Delta -1)e}$ we obtain $ \frac{1}{\sqrt{2\pi(h+1)}} \left(\frac{e(\Delta -1)}{h+1}+1\right)^{h+1}$. The second factor approaches $e^{e(\Delta-1)}$ as $h\to\infty$, and thus the whole expression tends to zero.

We now specialise to $\mathbb{Z}^n$. When allowing all simple paths, we have $\Delta= 2n$, which proves Proposition \ref{prop:lattice_no_perc} part (ii).

For non-backtracking paths, i.e.~part (i), by a union bound, there is positive probability of an increasing path on the whole lattice if and only if there is positive probability of an increasing path in the first (north-east) quadrant. In the first quadrant, non-backtracking paths are those whose co-ordinates do not decrease, and therefore the out-degree of each vertex is $n$. Thus by the same argument as above we have
\[\mathbb{P}( \exists \text{ an increasing path of length $h$}) \leq n^h \frac{(1+ \theta h)^{h+1}}{(h+1)!}\]
and again by the same argument as above this converges to zero if $\theta \le \frac{1}{ne}$.
\end{proof}

\begin{remark}
     While Proposition \ref{prop:lattice_no_perc} establishes extinction using the explicit bounds derived from Proposition \ref{prop:no_perco_GW}, one could alternatively frame this subcritical regime on the lattice using the multi-type branching process critical threshold $\lambda$ established in the previous section. Extinction is guaranteed when the associated maximal eigenvalue $\lambda_{\Delta-1,\theta}\le 1$.
\end{remark}

\subsection{Proof of Proposition \ref{prop:lattice_perc_l1}: RMF accessibility percolation with the graph distance via a simple coupling with oriented percolation}

Consider RMF percolation on $\mathbb{Z}^2$ with non-backtracking paths and using the graph distance $d(x,y) = \|x-y\|_1$. To establish that the critical threshold $\theta'_c$ is strictly smaller than 1, we employ a coupling argument with oriented percolation. The strategy is simple: recalling the RMF labelling from \eqref{eq:RMF}, we say that a vertex $v$ is \emph{open} if $U_v\in(0,\theta)$. We then observe that if a vertex in the first (north-east) quadrant is open, then by \eqref{eq:RMF}, its label $X_v$ is smaller than the vertices above and to the right; so an infinite up-right path of open vertices entails an infinite path whose labels are increasing. It is well-established that oriented site percolation has critical probability $p^\circ_c < 1$, and so we deduce that RMF percolation occurs for sufficiently large $\theta < 1$. We now provide the details.

\begin{proof}[Proof of Proposition \ref{prop:lattice_perc_l1}]
As set out above, we say that vertex $v$ is \emph{open} if $U_v\in(0,\theta)$. Suppose that $0=v_0,v_1,v_2,\ldots$ is a non-backtracking path, i.e.~$v_{j+1}$ is an up/right neighbour of $v_j$ for each $j=0,1,2,\ldots$. If $v_{j}$ is open, then
\[X_{v_{j+1}} = U_{v_{j+1}} + \theta d(0,v_{j+1}) \ge \theta (j+1) = \theta + \theta d(0,v_j) > U_{v_j} + \theta d(0,v_j) = X_{v_j}.\]
Thus if \emph{all} the vertices $v_0,v_1,v_2,\ldots$ are open, then the path is increasing, i.e.~$X_{v_0}<X_{v_1}<X_{v_2}<\ldots$ and accessibility percolation occurs. It therefore suffices to show that infinite open non-backtracking paths exist with positive probability.

Note that each vertex is open with probability $\theta$, independently of all other vertices, so existence of an infinite open non-backtracking path from $0$ is equivalent to the component of $0$ being infinite in oriented site percolation, as introduced in Section \ref{sec:zd_intro}, when the parameter $p$ is equal to $\theta$. Thus we deduce that $\theta'_c \le p^\circ_c$, where $p^\circ_c$ is the critical percolation parameter for oriented site percolation. Balister, Bollob\'as and Stacey \cite{balister1993upper} showed that $p^\circ_c \le 0.72599$, which completes the proof.
\end{proof}

This is essentially the same strategy as in \cite{duque2023rmf, hegarty2014existence}. The primary novelty of our work lies in extending this analysis to the more general $\ell^q$ case for $q>1$, which we carry out in the following section.

\subsection{Proof of Proposition \ref{prop:lattice_perc_l2}: RMF accessibility percolation with the $\ell^q$ distance for $q>1$ via the bricklayer coupling}\label{sec:bricklayer}

Writing $\mathbb{Z}_+ = \{0,1,2,\ldots\}$, we define the \emph{bricklayer lattice} to have vertices
\[V_{\mathbb{L}} = \{ (x,y) : y\in \mathbb{Z}_+,\, x-y/2 \in \mathbb{Z}_+\}\]
and directed edges
\[E_{\mathbb{L}} = \{ ((x,y),(x',y')) : x'=x+1 \text{ and } y=y', \text{ or } x'=x+1/2 \text{ and } y'=y+1\}. \]
We then create our bricks (see Figure \ref{fig:single_brick} for a visualisation): for each $(x,y)\in V_{\mathbb{L}}$ and $n\in 4\mathbb{N}$ (which we will eventually choose to be large), let
\[V^{(n)}_{x,y} = \{(a,b)\in\mathbb{Z}_+^2 : a\in \{xn, xn+1,\ldots,xn+n\},\, b\in\{2y, 2y+1, 2y+2\}.\]
To avoid dependence between edges and bricks, we only consider a subset of the possible directed edges within each brick. We have horizontal edges
\[\Hor = \{((a,b),(a',b')) \in (V_{x,y}^{(n)})^2 : a'=a+1 \text{ and either } b=b'=2y \text{ or } b=b'=2y+1\},\]
i.e.~the horizontal edges on the lower and middle layers of the brick. We then have the left-vertical edges
\[\LVer = \{((a,b),(a',b'))\in (V_{x,y}^{(n)})^2 : a'=a\in\{xn+\tfrac{n}{4},\ldots,xn+\tfrac{n}{2}-1\},\, b = 2y, \, b' = 2y+1\},\]
i.e.~the vertical edges between the lower layers of the brick that fall within the second horizontal quarter of the brick; and the right-vertical edges 
\[\RVer = \{((a,b),(a',b'))\in (V_{x,y}^{(n)})^2 : a'=a\in\{xn+\tfrac{n}{2},\ldots,xn+\tfrac{3n}{4}-1\},\, b = 2y+1, \, b' = 2y+2\},\]
i.e.~the vertical edges between the upper layers of the brick that fall within the third horizontal quarter of the brick.

\begin{figure}[h]
  \centering
  \includegraphics[width=0.7\textwidth]{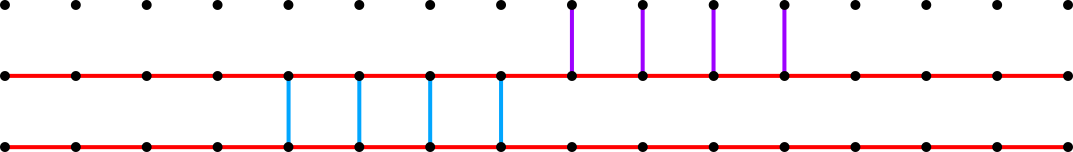}
  \caption{One brick, with the horizontal edges drawn in red, the left-vertical edges drawn in blue, and the right-vertical edges drawn in purple.}\label{fig:single_brick}
\end{figure}

We now introduce our coupling between the RMF labels (or, more precisely, the independent uniform random variables attached to each vertex) and an inhomogeneous and dependent oriented percolation model. We first do this in the case $q\in(1,\infty)$, and then later we will explain how to carry out the (easier) case $q=\infty$. Fix $q\in(1,\infty)$ and recall that $n\in 4\mathbb{N}$ will later be chosen to be large. For a horizontal oriented edge $((a,b),(a+1,b))$, say that the edge is open if $U_{(a,b)} \in (3\cdot\frac{5^q}{n^q},1-3\cdot\frac{5^q}{n^q})$. For a vertical oriented edge $((a,b),(a,b+1))$, say that the edge is open if $U_{(a,b)}<U_{(a,b+1)}$. Thus vertical edges are all open with probability $1/2$, whereas horizontal edges are open with probability close to $1$ when $n$ is large; and we note also that any two adjacent vertical edges are not independent, although the left-vertical and right-vertical edges within the bricks \emph{are} independent, since none of them are adjacent. We say that there is ``$n$-bricklayer percolation'' if there is an infinite path of open oriented edges in $\bigcup_{(x,y)\in V_{\mathbb{L}}} \left(\Hor\cup \LVer \cup \RVer\right)$. As mentioned above, this is an inhomogeneous and dependent form of oriented percolation.

For $(x,y)\in V_{\mathbb{L}}$, write $B_{x,y}^{(n)}$ for the brick defined above, i.e.~the oriented subgraph of $\mathbb{Z}_+^2$ consisting of vertices $V_{x,y}^{(n)}$ and edges $\Hor\cup \LVer \cup \RVer$. Note that the status of all edges within a brick $B_{x,y}^{(n)}$ are independent, and (for a fixed $n$) independent of all edges of other bricks $B_{x',y'}^{(n)}$, $(x',y')\in V_{\mathbb{L}}$. We say that a brick is \emph{good} if all its horizontal edges are open, at least one left-vertical edge is open, and at least one right-vertical edge is open. We then make three observations:
\begin{enumerate}
\item The probability that a brick is good satisfies
\[\P\left(B^{(n)}_{x,y} \text{ is good}\right) = \left(1-6\cdot\frac{5^q}{n^q}\right)^{2n} \big(1-2^{-n/4}\big)^2.\]
This follows from the facts that each horizontal edge is open with probability $1-6\cdot 5^q/n^q$ if $q\in(1,\infty)$, each vertical edge is open with probability $1/2$, and as mentioned above, edges within a brick are open independently of each other.
\item The events
\[\Big\{ \big\{B^{(n)}_{x,y} \text{ is good}\big\} : (x,y)\in V_{\mathbb{L}}\Big\}\]
are independent. This follows from the fact that, for fixed $n$, the status of edges in one brick is independent of the status of edges in any other brick.
\item A directed path of good bricks entails a path of directed open edges starting at the bottom-left corner of the first brick in the path, and ending at the middle-right vertex of the last brick in the path. See Figure \ref{fig:bricks_colour}.
\end{enumerate}

\begin{figure}[h]
  \centering
  \includegraphics[width=\textwidth]{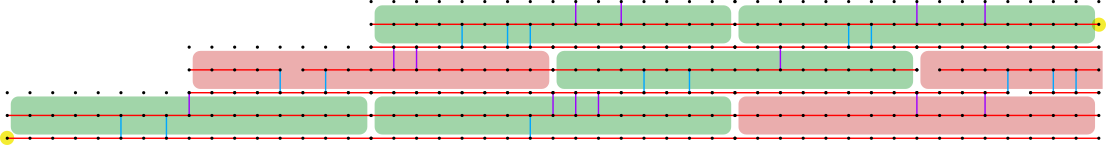}
  \caption{A small portion of the bricklayer lattice with each vertex replaced by a brick. Only open edges are drawn. Good bricks are coloured green and bad bricks are coloured red. A directed path of good bricks guarantees a path of open edges from the bottom-left of the first brick to the middle-right vertex of the last brick, both highlighted in yellow.}\label{fig:bricks_colour}
\end{figure}

\begin{lemma}\label{lem:inhom_perco}
The probability of $n$-bricklayer percolation tends to $1$ as $n\to\infty$.
\end{lemma}

\begin{proof}
We note that since $q>1$, from our first observation above, the probability that a brick is good converges to $1$ as $n\to\infty$. By our second observation, each brick is independently good. Furthermore, the bricklayer lattice $(V_{\mathbb{L}}, E_{\mathbb{L}})$ is homeomorphic to the graph consisting of vertices $\mathbb{Z}_+^2$ with all up and right edges, i.e.~the first quadrant of the usual oriented square lattice. Since oriented site percolation on $\mathbb{Z}_+^2$ has an infinite component with probability tending to $1$ as the percolation probability converges to $1$, we deduce that there exists a directed path of good bricks with probability tending to $1$ as $n\to\infty$. Finally, by our third observation above, a directed path of good bricks entails existence of an infinite path of directed open edges starting from $(0,0)$ that remains within the union of all the (good) bricks; that is, there is $n$-bricklayer percolation.
\end{proof}

The final observation that we would \emph{like} to make is that if $\theta$ is sufficiently close to $1$, open edges in the inhomogeneous dependent oriented percolation model correspond to edges along which the RMF labels are increasing. If this were true, then an infinite open path in the percolation model would entail an infinite increasing path in RMF accessibility percolation. It is, unfortunately, not necessarily true for \emph{all} edges; but it is true for those within bricks sufficiently far to the right. This will be enough to complete the proof by extending the path to $0$ via an additional finite segment.

\begin{lemma}\label{lem:pdist}
For $q\in(1,\infty)$ and $n>5$, choose $A_q(n)$ large enough that for all $a\ge A_q(n)$,
\[\left(1+\frac{q}{a}\right)^{1/q} \ge 1 + \left(1-\frac{5^q}{n^q}\right)\frac{1}{a}.\]
Suppose that $(x,y)\in V_{\mathbb{L}}$ with $x\ge 2$, and suppose that $(a,b)\in V^{(n)}_{x,y}$ satisfies $a\ge A_q(n)$. Then
\[d(0,(a+1,b)) - d(0,(a,b)) > 1 - 2\cdot\frac{5^q}{n^q}.\]
\end{lemma}

\begin{remark}
The choice of $5^q/n^q$ when we choose $A_q(n)$, and in our definition of bricklayer percolation, is largely arbitrary. We could choose $A_q(n)$ such that the right-hand side of the first equation above is arbitrarily close (in $n$) to $1+1/a$, and that would give us a better bound in the second equation. However, we will see in the proof below that an error of size $(5/n)^q$ emerges from a rough geometric bound, so choosing $A_q(n)$ to allow another error of the same size simply minimizes notation.
\end{remark}

\begin{proof}
We have
\[d(0,(a+1,b)) = \left(a^q(1+\tfrac{1}{a})^q + b^q\right)^{1/q} \ge \left( a^q + qa^{q-1} + b^q\right)^{1/q} = \left(a^q + b^q\right)^{1/q}\left(1 + \frac{qa^{q-1}}{a^q+b^q}\right)^{1/q}.\]
Since $a \ge A_q(n)$, we also have $(a^q+b^q)/a^{q-1} \ge A_q(n)$, and we obtain from above that
\[d(0,(a+1,b)) \ge \left(a^q + b^q\right)^{1/q}\left( 1 + \left(1-\frac{5^q}{n^q}\right)\frac{a^{q-1}}{a^q + b^q} \right).\]
Thus
\begin{equation}\label{eq:first_dist_bd}
d(0,(a+1,b)) - d(0,(a,b)) \ge \left( 1-\frac{5^q}{n^q} \right)\frac{a^{q-1}}{(a^q + b^q)^{1-1/q}} = \left( 1-\frac{5^q}{n^q} \right)\frac{1}{(1 + (b/a)^q)^{1-1/q}}.
\end{equation}
Now, since $(a,b)\in V^{(n)}_{x,y}$, we have $b\le 2y+2$ and $a\ge nx$. Since $(x,y)\in V_{\mathbb{L}}$, we have $y\le 2x$, and therefore $b\le 4x+2 \le 5x \le 5a/n$, where we used that $x\ge 2$ for the second inequality. Substituting this into \eqref{eq:first_dist_bd}, we have
\begin{align*}
d(0,(a+1,b)) - d(0,(a,b)) &\ge \left( 1-\frac{5^q}{n^q} \right)\frac{1}{(1 + 5^q/n^q)^{1-1/q}}\\
&\ge \left( 1-\frac{5^q}{n^q} \right)\frac{1}{(1 + 5^q/n^q)} \ge \left( 1-\frac{5^q}{n^q} \right)^2
\end{align*}
and the result follows.
\end{proof}

\begin{corollary}\label{cor:open_incr}
Suppose that $q\in(1,\infty)$ and $1-5^q/n^q<\theta<1$. Suppose that $(x,y)\in V_{\mathbb{L}}$ with $x\ge 2$, and take $(a,b)\in V^{(n)}_{x,y}$ such that $a\ge A_q(n)$, where $A_q(n)$ was defined in Lemma \ref{lem:pdist}. If the oriented edge from $(a,b)$ to $(a+1,b)$ is open in the inhomogeneous dependent oriented percolation model, then its RMF labels are increasing; that is, $X_{(a+1,b)} > X_{(a,b)}$.
\end{corollary}

\begin{proof}
We recall that the edge from $(a,b)$ to $(a+1,b)$ is open if
\[U_{(a,b)} \in \left(3\cdot\frac{5^q}{n^q},\,1-3\cdot\frac{5^q}{n^q}\right).\]
Thus, applying Lemma \ref{lem:pdist},
\begin{align*}
X_{(a+1,b)} - X_{(a,b)} &= U_{(a+1,b)} + \theta d(0,(a+1,b)) - U_{(a,b)} - \theta d(0,(a,b))\\
&\ge 0 + \theta d(0,(a+1,b)) - 1+3\cdot\frac{5^q}{n^q} - \theta d(0,(a,b))\\
&>  \theta \left(1 - 2\cdot\frac{5^q}{n^q}\right) - 1+3\cdot\frac{5^q}{n^q}.
\end{align*}
By our choice of $\theta$, this is strictly positive, completing the proof.
\end{proof}

\begin{proof}[Proof of Proposition \ref{prop:lattice_perc_l2}]
We first do the case $q\in(1,\infty)$. By Lemma \ref{lem:inhom_perco}, for any $\eps>0$ we can choose $n$ large enough that the probability of $n$-bricklayer percolation is larger than $1-\eps$. This gives us an infinite path of open edges in the inhomogeneous dependent oriented percolation model that remains within the union of the bricks. By Corollary \ref{cor:open_incr}, if we take $\theta\in(1-5^q/n^q,1)$ then for all except finitely many of the open horizontal edges, the RMF labels are increasing across the edge; and directly from the definition of our coupling, any open vertical edge has RMF labels increasing across the edge. We therefore have an infinite increasing path started from some point at distance at most $2A_q(n)$ from $0$. By further increasing $\theta$ if necessary, we can ensure that every oriented edge within distance $2A_q(n)$ of $0$ has increasing RMF labels. This completes the proof in the case $q\in(1,\infty)$.

In the case $q=\infty$, things are slightly simpler since every horizontal edge below the diagonal $x=y$ moves us distance exactly $1$ further from the origin. We make a rather arbitrary choice of coupling, saying that a horizontal oriented edge from $(a,b)$ to $(a+1,b)$ is open if $U_{(a,b)} \in (n^{-2},1-n^{-2})$, and as in the $q\in(1,\infty)$ case, a vertical oriented edge from $(a,b)$ to $(a,b+1)$ is open if $U_{(a,b)} < U_{(a,b+1)}$. We then have
\[\P\left(B^{(n)}_{x,y} \text{ is good}\right) = \left(1-2n^{-2}\right)^{2n} \big(1-2^{-n/4}\big)^2\]
and therefore the probability that a brick is good converges to $1$ as $n\to\infty$. Our other observations before Lemma \ref{lem:inhom_perco} still hold, and therefore Lemma \ref{lem:inhom_perco} goes through without further changes. Then as above, for any $\eps>0$, we can choose $n\ge 5$ large enough that the probability of $n$-bricklayer percolation is larger than $1-\eps$. Take $\theta\in(1-n^{-2},1)$. Then by the same argument as in Lemma \ref{lem:pdist}, for any $(x,y)\in V_{\mathbb{L}}$ with $x\ge 2$ and $n\ge 5$, if $(a,b)\in V_{x,y}^{(n)}$ then $b\le a$ and therefore
\[d(0,(a+1,b)) - d(0,(a,b)) = 1.\]
Then as in the proof of Corollary \ref{cor:open_incr}, if the oriented edge from $(a,b)$ to $(a+1,b)$ is open, we have
\begin{align*}
X_{(a+1,b)} - X_{(a,b)} &= U_{(a+1,b)} + \theta d(0,(a+1,b)) - U_{(a,b)} - \theta d(0,(a,b))\\
&\ge 0 + \theta d(0,(a+1,b)) - 1+n^{-2} - \theta d(0,(a,b))\\
&=  \theta  - 1+n^{-2} > 0
\end{align*}
by our choice of $\theta$, so the RMF labels are increasing across the edge. And again directly from the definition of the coupling, any open vertical edge has RMF labels that increase across the edge. The remainder of the proof proceeds as in the case $q\in(1,\infty)$.
\end{proof}

\section*{Acknowledgments}
DDAB was supported by a scholarship from the EPSRC Centre for Doctoral Training in Statistical Applied Mathematics at Bath (SAMBa), under the project EP/S022945/1. MR was supported in part by a Royal Society University Research Fellowship, grant URF\textbackslash R\textbackslash 211038, and in part by EPSRC grant EP/X040089/1. Both authors are very grateful to Andreas Kyprianou, C\'ecile Mailler and Bastien Mallein for helpful conversations.

\bibliographystyle{plain}

\end{document}